\documentclass[onefignum,onetabnum]{siamart171218} 

\usepackage{lipsum}
\usepackage{amsfonts}
\usepackage{graphicx}
\usepackage{epstopdf}
\usepackage{mathrsfs}
\usepackage{changes}
\usepackage{tikz-cd}
\usepackage{algorithmic}

\ifpdf
  \DeclareGraphicsExtensions{.eps,.pdf,.png,.jpg}
\else
  \DeclareGraphicsExtensions{.eps}
\fi

\newsiamremark{remark}{Remark}
\newsiamremark{hypothesis}{Hypothesis}
\crefname{hypothesis}{Hypothesis}{Hypotheses}
\newsiamthm{claim}{Claim}
\newsiamremark{assumption}{Assumption}

\newcommand{\la}{\langle}
\newcommand{\ra}{\rangle}

\newcommand{\ud}{\,\mathrm{d}}

\newcommand{\bR}{\mathbb{T}^{d}}

\newcommand{\eps}{\varepsilon}

{}

\headers{Homogenization of Time-Discrete Gradient Flows}{Y. Deng, Y. Gao, and L. Wang}

\title{Homogenization of Time-Discrete Gradient Flows\thanks{Submitted to the editors DATE.
\funding{LW is partially funded by NSF grant DMS-2513336 and Simons Foundation TSM. YG is partially funded by NSF grant DMS-2440651.}}}

\author{Yunhong Deng\thanks{School of Mathematics, University of Minnesota,
Minneapolis, MN 55455 (\email{deng0335@umn.edu}, \email{liwang@umn.edu}).}
\and Yuan Gao\thanks{Department of Mathematics, Purdue University, West Lafayette, IN 47907 (\email{gao662@purdue.edu}).}
\and Li Wang\footnotemark[2]}

\usepackage{amsopn}

\makeatletter
\newcommand*{\addFileDependency}[1]{
  \typeout{(#1)}
  \@addtofilelist{#1}
  \IfFileExists{#1}{}{\typeout{No file #1.}}
}
\makeatother

\newcommand*{\myexternaldocument}[1]{%
    \externaldocument{#1}%
    \addFileDependency{#1.tex}%
    \addFileDependency{#1.aux}%
}

\ifpdf
\hypersetup{
  pdftitle={Homogenization of Time-Discrete Gradient Flows},
  pdfauthor={Y. Deng, G. Yuan, and L. Wang}
}
\fi
\myexternaldocument{ex_supplement}

\begin{document}

\maketitle


\begin{abstract}In this work, we study the homogenization limit for Fokker--Planck equations on the flat torus with rapidly oscillating diffusion coefficients at the time-discrete level. We discretize the equation in time using a minimizing movement scheme and compare two choices of metric: a transport-type metric and a weighted $L^2$ metric. We prove that this choice is crucial: the two schemes lead to different homogenization limits as the scaling parameter tends to zero, and consequently to different evolution equations as the time step is sent to zero. In particular, only the scheme based on the weighted $L^2$ metric recovers the correct effective dynamics.
\end{abstract}

\begin{keywords}
homogenization, Fokker-Planck equation, gradient flow, minimizing movements.
\end{keywords}

\begin{AMS}
35A15, 35B27, 35Q84, 37M15, 46N40, 49Q22.
\end{AMS}

\section{Introduction}
This paper deals with the homogenization of two time-discrete approximations of Fokker-Planck equations on $d$-dimensional flat torus $\mathbb{T}^{d} := \mathbb{R}^{d}/\mathbb{Z}^{d}$. Given (non-constant) matrix field $A(x)$, and confinement potential $V_{\varepsilon}(x)$, we consider the $\varepsilon$-Fokker-Planck equation:
\begin{equation}\label{0}
    \frac{\partial \mu_{\varepsilon}}{\partial t} - \nabla \cdot \Big[A\Big(\frac{x}{\varepsilon}\Big) \big(\nabla \mu_{\varepsilon} + \mu_{\varepsilon} \nabla V_{\varepsilon}(x)\big)\Big] = 0\,,  \quad \text{for}~(x,t) \in  \mathbb{T}^{d} \times (0, \infty),
\end{equation}
where $0 < \varepsilon \ll 1$ is the ratio between the micro and macro scales of the system.

\begin{assumption}\rm\label{assumption}
We assume the model~\eqref{0} obeys the following.
\begin{enumerate}
    \item $A : \mathbb{R}^{d} \to \mathbb{R}^{d \times d}$ is $1$-periodic such that $A(x + e_{j}) = A(x)$ for all $x \in \mathbb{R}^{d}$, where $(e_{j})_{j = 1}^{d}$ are the standard basis vectors of $\mathbb{R}^{d}$.\label{A1}
    \item There exist positive constants $\Lambda_{0}, \Lambda_{1}$ such that $\Lambda_{0}I_{d \times d} \preceq A(x) \preceq \Lambda_{1} I_{d \times d}$ for all $x \in \mathbb{R}^{d}$.\label{A3}
    \item $A$ is smooth (at least of class $C^{2}$).
    \item $V_{\varepsilon}(x) \ge 0$ for all $x \in \mathbb{T}^{d}$ and $\varepsilon > 0$.\label{A5} 
    \item $V_{\varepsilon} \in C^{1}(\mathbb{T}^{d})$ and there exists a positive constant $\lambda$ such that $V_{\varepsilon}(x) \le \lambda$ and $|\nabla V_{\varepsilon}(x)| \le \lambda$ for all $x \in \mathbb{T}^{d}$ and $\varepsilon > 0$.\label{A4}
    \item There exists a limiting potential $V$ such that $\sup_{x \in \mathbb{T}^{d}}|V_{\varepsilon}(x) - V(x)| \to 0$ as $\varepsilon \to 0$.\label{A6}
\end{enumerate}
\end{assumption}

Under the assumptions above, the asymptotic analysis in homogenization theory (see, for example, \cite{allaire2007homogenization} and \cite[Appendix A]{gao2023homogenization}) suggests that solutions of the $\varepsilon$-Fokker-Planck equation~\eqref{0} converge, as $\varepsilon \to 0$, to solutions of the following effective Fokker-Planck equation,
\begin{equation}\label{effect}
    \frac{\partial \mu}{\partial t} - \nabla \cdot \Big[A_{\ast}\big(\nabla \mu + \mu\nabla V(x)\big)\Big] = 0 \text{ in } \mathbb{T}^{d} \times (0, \infty),
\end{equation}
where $A_{\ast}$ is the effective diffusion coefficient given by
\begin{equation}\label{B-hom}
    (A_{\ast})_{jk} = \int_{\mathbb{T}^{d}} A_{jk}(x) + \sum_{\ell = 1}^{d} A_{j\ell} \frac{\partial\chi_{k}(x)}{\partial x_{\ell}} \ud x,
\end{equation}
$\chi(x) \in \mathbb{R}^{d}$ is a corrector defined by the periodic solution of $-\nabla \cdot (A(x)\nabla \chi) = \nabla \cdot A(x)$ up to an additive constant, and
\begin{equation}\label{B-hom-1}
    \Lambda_{0}I_{d \times d} \preceq A_{\ast} \preceq \Lambda_{1} I_{d \times d}.
\end{equation}

Our starting point is the recent work by Gao and Yip \cite{gao2023homogenization}, in which they investigated the homogenization of~\eqref{0} from the perspective of Wasserstein gradient flows. Let $H_{\varepsilon}(\mu) = \int_{\mathbb{T}^{d}} \log(\mu) + V_{\varepsilon}(x) \ud \mu$. By writing $\varepsilon$-Fokker-Planck equation~\eqref{0} in the form:
\begin{equation}\label{0-copy}
    \frac{\partial \mu_{\varepsilon}}{\partial t} - \nabla\cdot\bigg[\mu_{\varepsilon} A\Big(\frac{x}{\varepsilon}\Big) \nabla \frac{\delta H_{\varepsilon}(\mu_{\varepsilon})}{\delta \mu_{\varepsilon}}\bigg] = 0 \text{ in } \mathbb{T}^{d} \times (0, \infty),
\end{equation}
one can view~\eqref{0} as the gradient flow of $H_{\varepsilon}(\mu)$ with respect to the $\varepsilon$-Wasserstein distance in the form of Benamou-Brenier formula as follows:
\begin{equation*}
    W_{\varepsilon}(\mu_{0}, \mu_{1})^{2} := \min_{\mu, v} \int_{0}^{1}\int_{\mathbb{T}^{d}} \big<A(x/\varepsilon)^{-1} v_{t}(x), v_{t}(x)\big> \ud \mu_{t} \ud t\,.
\end{equation*}
Here the minimum runs over all absolutely continuous curves $(\mu_{t})_{0 \le t \le 1}$ in the space of probability measures, and velocity fields $(v_{t})_{0 \le t\le 1}$ satisfying the continuity equation $\partial_{t}\mu_{t} + \nabla \cdot(\mu_{t}v_{t}) = 0$ in the sense of distribution with boundary conditions $\mu_{t = 0} = \mu_{0}$ and $\mu_{t = 1} = \mu_{1}$. For detailed discussions related to the gradient-flow formulation, we refer the reader to \cite{forkert2022evolutionary, gao2023homogenization}. This gradient flow formulation in Wasserstein spaces has yielded rich analytical results (see, for example, \cite{carrillo2003kinetic, carrillo2006contractions, carrillo2011global, ambrosio2008gradient, villani2008optimal}) for Fokker--Planck-type and aggregation--diffusion-type equations since seminal works of Jordan, Kinderlehrer and Otto \cite{jordan1998variational}, and Otto \cite{otto2001geometry}.

Based on the energy dissipation inequality (EDI) form of gradient flows and the method of $\Gamma$-convergence of gradient flows \cite{sandier2004gamma, serfaty2011gamma}, Gao and Yip \cite{gao2023homogenization} showed that the homogenization limit of~\eqref{0-copy} is the gradient flow of $H(\mu) = \int_{\mathbb{T}^{d}} \log(\mu) + V(x) \ud \mu$ with respect to an effective Wasserstein distance
\begin{equation*}
    W_{\ast}(\mu_{0}, \mu_{1})^{2} := \min_{\mu, v} \int_{0}^{1}\int_{\mathbb{T}^{d}} \la A_{\ast}^{-1} v_{t}(x), v_{t}(x) \ra \ud \mu_{t} \ud t,
\end{equation*}
where $A_{\ast}$ is the effective diffusion coefficient in~\eqref{B-hom}, and the minimum in the distance runs over the same constraints as that in the $\varepsilon$-Wasserstein distance above. Therefore, it indicates that the limiting Wasserstein gradient flow is consistent with the effective Fokker-Planck equation~\eqref{effect}. 

However, it is important to note that if one does not consider the dynamics, but instead studies only the homogenization limit of the metric $W_\varepsilon$, then the desired effective distance is not recovered. More precisely, as shown in \cite{gao2023homogenization},
\begin{equation*}
W_{\ast}(\mu,\nu)
\neq
\lim_{\varepsilon\to 0} W_{\varepsilon}(\mu,\nu)=: W_{\mathrm{GH}}(\mu, \nu)\,,
\end{equation*}
where $W_{\mathrm{GH}}(\mu, \nu)$ can be computed using the Gromov--Hausdorff convergence; see \cite[Theorem 28.6]{villani2008optimal} and \cite[Section 5.2]{gao2023homogenization}.

Our main interest is to determine whether the natural time-discrete counterpart of this gradient-flow formulation, namely the JKO scheme \cite{jordan1998variational}, preserves the homogenization limit~\eqref{effect} at the semi-discrete level. Through a detailed analysis, we find that the JKO scheme is not asymptotic-preserving: its homogenized limit differs from the effective limit~\eqref{effect}. Instead, as $\varepsilon\to 0$, it converges to the gradient flow of $H(\mu)$ with respect to the metric $W_{\mathrm{GH}}$. 
In contrast, if \eqref{0} is viewed as a gradient flow with respect to a weighted $L^2$ metric, then the corresponding minimizing movement scheme is asymptotic-preserving and recovers the correct effective limit~\eqref{effect}.

We now provide rigorous definitions of the two time-discrete schemes and summarize our main results concerning their homogenization limits for a fixed time step and the convergence in time of the resulting limit schemes.

\subsection{Variational structure with weighted $L^{2}$ metric} We first rewrite~\eqref{0} as a heat-type equation (see for example \cite[equation (1.9)]{liu2018positivity}, \cite[equation (3.2)]{gao2023homogenization} and \cite{carrillo2025positivity}): 
\begin{equation}\label{eqb}
    \begin{aligned}
        \frac{\partial \mu_{\varepsilon}}{\partial t} - \nabla\cdot\Big[\pi_{\varepsilon}(x)A\Big(\frac{x}{\varepsilon}\Big)\nabla \frac{\mu_{\varepsilon}}{\pi_{\varepsilon}}\Big] = 0\,, \qquad \text{where}\quad  \pi_{\varepsilon}(x) := e^{-V_{\varepsilon}(x)}\,.
    \end{aligned}
\end{equation}
Motivated by the fact that the heat equation is the $L^{2}$ gradient flow of the Dirichlet energy $\int_{\mathbb{T}^{d}} |\nabla \mu|^{2} \ud x$ \cite{strikwerda2004finite, jordan1998variational}, we introduce the following weighted $L^{2}$ scheme for~\eqref{eqb}.

\begin{definition}[weighted $L^{2}$ scheme]\label{L2-scheme-def} Given $\varepsilon, \tau$ and initial value $\mu_{0} \in H^{1}(\mathbb{T}^{d})$, we define the solution of the weighted $L^{2}$ scheme as follows:
\begin{equation}\label{0708}
    \mu_{\varepsilon, \tau} : (0, \infty) \to H^{1}(\mathbb{T}^{d}),\quad \mu_{\varepsilon, \tau}(t) = \mu^{n}_{\varepsilon} \quad \text{ for }~~ (n - 1)\tau < t \le n\tau,
\end{equation}
where $\mu^{0}_{\varepsilon} = \mu_{0}$ and $\mu_{\varepsilon}^{n}$ is defined iteratively by
\begin{equation}\label{L2-variational}
    \mu_{\varepsilon}^{n} = \arg\min_{\mu} Q_{\varepsilon}(\mu) + \frac{1}{2\tau} K_{\varepsilon}(\mu, \mu_{\varepsilon}^{n - 1})^{2} \text{ over } \mu \in H^{1}(\mathbb{T}^{d}).
\end{equation}
Here $Q_{\varepsilon}$ and $K_{\varepsilon}$ are the weighted Dirichlet energy and the weighted $L^{2}$ distance, respectively: 
\begin{align}
    Q_{\varepsilon}(\mu) := \frac{1}{2}\int_{\mathbb{T}^{d}} \Big<\pi_{\varepsilon}(x)A\Big(\frac{x}{\varepsilon}\Big)\nabla \varrho_\varepsilon, \nabla \varrho_\varepsilon\Big> \ud x,\quad K_{\varepsilon}(\mu, \nu) := \bigg[\int_{\mathbb{T}^{d}} \frac{(\mu - \nu)^{2}}{\pi_{\varepsilon}} \ud x\bigg]^{\frac{1}{2}},\label{disK}
\end{align}
where $\varrho_\varepsilon  := \mu/\pi_{\varepsilon}$. 
\end{definition}

The well-posedness and convergence of this scheme will be established in Section~\ref{sec2}. We assume without loss of generality that $\int_{\mathbb{T}^{d}} \mu_{0} \ud x = 1$ and $\mu_{0} \ge 0$. It is then straightforward to check that solutions to the weighted $L^{2}$ scheme remain in the space of probability measures, i.e.,
\begin{equation}\label{structure-preserving}
    \int_{\mathbb{T}^{d}} \mu_{\varepsilon, \tau}(x, t) \ud x = 1 ~\text{ and }~ \mu_{\varepsilon, \tau}(t) \ge 0 \text{ a.e. }
\end{equation}
for all $t \in (0, \infty)$.

\subsection{Variational structure with transport metric}
Let $P(\mathbb{T}^{d})$ be the space of probability measures on $\mathbb{T}^{d}$. We introduce the following $\varepsilon$-JKO scheme.
\begin{definition}[$\varepsilon$-JKO scheme]\label{JKO-scheme-def}Given $\varepsilon, \tau$ and initial value $\mu_{0} \in P(\mathbb{T}^{d})$ satisfying $S(\mu_{0}) < \infty$, we define the solution of the $\varepsilon$-JKO scheme as follows: 
\begin{equation}\label{JKO-solution}
    \mu_{\varepsilon, \tau} : (0, \infty) \to P(\mathbb{T}^{d}),\quad \mu_{\varepsilon, \tau}(t) = \mu^{n}_{\varepsilon} \quad \text{ for } ~~(n - 1)\tau < t \le n\tau,
\end{equation}
where $\mu^{0}_{\varepsilon} = \mu_{0}$ and $\mu^{n}_{\varepsilon}$ is defined iteratively by
\begin{equation} \label{JKO-WW}
    \mu^{n}_{\varepsilon} = \arg\min_{\mu} H_{\varepsilon}(\mu) + \frac{1}{2\tau} W_{\varepsilon}(\mu, \mu^{n - 1}_{\varepsilon})^{2} \text{ over } \mu \in P(\mathbb{T}^{d}).
\end{equation}
Here $H_{\varepsilon}$ and $W_{\varepsilon}$ are the free energy and the $\varepsilon$-Wasserstein distance, respectively:
\begin{equation}\label{Hvare}
    \begin{aligned}
        H_{\varepsilon}(\mu) &:= S(\mu) + \int_{\mathbb{T}^{d}} V_{\varepsilon}(x) \ud \mu,\\
        W_{\varepsilon}(\mu, \nu) &= \inf \bigg\{\int_{\mathbb{T}^{d} \times \mathbb{T}^{d}} c_{\varepsilon}(x, y) \ud \omega(x, y) : \omega \in \Pi(\mu, \nu)\bigg\}^{\frac{1}{2}},
    \end{aligned}
\end{equation}
where $S$ is the log-entropy functional defined by
\begin{equation*}
    S(\mu) := \begin{cases}
            \int_{\mathbb{T}^{d}} \varrho(x) \log(\varrho(x))\ud x &\text{ if } \mathrm{d}\mu = \varrho(x)\ud x\\
            \infty &\text{ otherwise }
        \end{cases}\,.
\end{equation*}
$\Pi(\mu, \nu)$ is the set of couplings between $\mu$ and $\nu$, i.e., the set of joint probability distributions, with marginal distribution $\mu$ and $\nu$ in each variable, and $c_\eps(x,y)$ is a cost function defined by the following least action,
\begin{equation}\label{eps-c}
    \begin{aligned}
        c_\eps(x,y) = \min_{z}\bigg\{\int_0^1 \big<A(z/\varepsilon)^{-1}\dot{z},\, \dot{z}\big> \ud t: z(0) = x,~z(1) = y, \text{ and } z \in H^{1}\big([0, 1]; \mathbb{T}^{d}\big)\bigg\},
    \end{aligned}
\end{equation}
where $\dot{z}$ is the time derivative of the curve $z$, and $H^{1}\big([0, 1]; \mathbb{T}^{d}\big)$ is the space of Sobolev curves.
\end{definition}

The well-posedness and convergence of $\varepsilon$-JKO scheme will be established in Section~\ref{sec3}. We refer readers to Section~\ref{sec1-1} below for a preliminary discussion on optimal transport.

\subsection{Homogenization}
Our first results concern the homogenization of the two schemes.
\begin{theorem}[weighted $L^{2}$ scheme]\label{homo-of-ie}Let $\mu_{\varepsilon, \tau}$ be the solution of the weighted $L^{2}$ scheme in Definition~\ref{L2-scheme-def} with initial value $\mu_{0} \in H^{1}(\mathbb{T}^{d})$, we have
\begin{equation*}
    \mu_{\varepsilon, \tau}(t) \to \mu_{\tau}(t) \text{ in } L^{2}(\mathbb{T}^{d}) \text{ as } \varepsilon \to 0 ~~ \text{ for all }~ t \in (0, \infty) \,,
\end{equation*}
where $\mu_{\tau}$ is the the homogenized solution with initial value $\mu_{0}$, and 
\begin{align} \label{L2-mu}
    \mu_{\tau}(t) = \mu^{n} \quad \text{for } (n - 1)\tau < t \le n\tau\,.
\end{align}
Here $\mu^n$ is iteratively defined by
\begin{equation}\label{limit-BE}
    \mu^{n} = \arg\min_{\mu} Q(\mu) + \frac{1}{2\tau}K(\mu, \mu^{n - 1})^{2} \text{ over } \mu \in H^1(\mathbb{T}^{d}),
\end{equation}
where 
\begin{equation}\label{limit-I}
    \begin{aligned}
        Q(\mu) := \frac{1}{2}\int_{\mathbb{T}^{d}}\big<\pi_{\ast}(x) A_{\ast}\nabla \varrho, \nabla \varrho\big> \ud x, \quad K(\mu, \nu) := \bigg[\int_{\mathbb{T}^{d}} \frac{(\mu - \nu)^{2}}{\pi_{\ast}} \ud x\bigg]^{\frac{1}{2}},
    \end{aligned}
\end{equation}
and $\varrho := \mu/\pi_{\ast}$ and $\pi_{\ast} := e^{-V}$. $A_{\ast}$ is the effective diffusion constant given in~\eqref{B-hom}.
\end{theorem}

\begin{theorem}[$\varepsilon$-JKO scheme]\label{homo-of-jko}Let $\mu_{\varepsilon, \tau}$ be the solution of the $\varepsilon$-JKO scheme in Definition~\ref{JKO-scheme-def} with initial value $\mu_{0} \in P(\mathbb{T}^{d})$, we have
\begin{equation}\label{homo-jko-solution-def}
    \mu_{\varepsilon, \tau}(t) \to \overline{\mu}_{\tau}(t) \text{ in $W_{2}$ distance}~~ \text{ as } ~\varepsilon \to 0  \quad \text{ for all }~ t \in (0, \infty)
\end{equation}
where $\overline{\mu}_{\tau}$ is the homogenized solution with initial value $\mu_{0}$, and
\begin{align*}
     \overline{\mu}_{\tau}(t) = \overline{\mu}^{n} \quad \text{for } (n - 1)\tau < t \le n\tau\,.
\end{align*}
Here $\overline{\mu}^n$ is iteratively defined by
\begin{equation}\label{limiting-JKO-def}
    \overline{\mu}^{n} = \arg\min_{\mu} H(\mu) + \frac{1}{2\tau} W_{\mathrm{GH}}(\mu, \overline{\mu}^{n - 1})^{2} \text{ over } \mu \in P(\mathbb{T}^{d}).
\end{equation}
$H$ and $W_{\mathrm{GH}}$ are the homogenized entropy and homogenized distance, respectively:
\begin{equation}\label{limiting_W}
    \begin{aligned}
        H(\mu) := S(\mu) &+ \int_{\mathbb{T}^{d}} V(x) \ud \mu,\\
        &W_{\mathrm{GH}}(\mu, \nu) := \inf \bigg\{\int_{\mathbb{T}^{d} \times \mathbb{T}^{d}} c_{\mathrm{hom}}(x, y) \ud \omega(x, y) : \omega \in \Pi(\mu, \nu)\bigg\}^{\frac{1}{2}},
    \end{aligned}
\end{equation}
where the point-wise limiting cost function $c_{\mathrm{hom}}(x,y)$ is defined by
\begin{equation*}
    c_{\mathrm{hom}}(x,y) = \min_{z}\bigg\{\int_{0}^{1}A_{\mathrm{hom}}(\dot{z})\ud t: z(0) = x,~z(1)=y, \text{ and } z\in H^{1}\big([0, 1]; \mathbb{T}^{d}\big)\bigg\},
\end{equation*}
and
\begin{equation}\label{Bast}
    \begin{aligned}
        A_{\mathrm{hom}}(\xi) := \liminf_{T \to \infty}\, \inf_{\varphi \in H^{1}_{0}([0, T]; \mathbb{R}^{d})} \frac{1}{T}\int_0^T \big<A(t\xi + \varphi)^{-1}(\xi + \dot{\varphi}), \xi + \dot{\varphi}\big> \ud t.
    \end{aligned}
\end{equation}
\end{theorem}

\begin{remark}\rm
This action function $A_{\mathrm{hom}}$ comes from the homogenization of Riemannian metrics (see, for example, \cite[Theorem 2.6]{braides2002riemannian}). From \cite{weinan1991class, braides2002riemannian}, it obeys
\begin{equation}\label{B-ast-1}
    \Lambda_{1}^{-1}|\xi|^{2} \le A_{\mathrm{hom}}(\xi) \le \Lambda_{0}^{-1}|\xi|^{2},\quad \text{ for all } \xi \in \mathbb{R}^{d},
\end{equation}
where $\Lambda_{0}$ and $\Lambda_{1}$ are the lower bound and the upper bound of $A(x)$ given in Assumption~\ref{A3}.
\end{remark}

\subsection{Time-continuous limit}
To determine whether the homogenized schemes are consistent with the time-continuous homogenization limit~\eqref{effect}, we identify the time-continuous limits of the homogenized solutions as $\tau\to 0$.

\begin{theorem}\label{prop-homo-l2}
Let $\mu_{\tau}$ be defined from \eqref{L2-mu}. Then $(\mu^n)_{n \in \mathbb{N}}$ solves 
\begin{equation*}
    \frac{\mu^{n} - \mu^{n - 1}}{\tau} - \nabla \cdot \Big[A_{\ast}\big(\nabla \mu^{n} + \mu^{n}\nabla V(x)\big)\Big] = 0,
\end{equation*}
in the $H^{1}$ weak sense. Moreover, we have 
\begin{equation*}
    \mu_{\tau}(t) \to \mu(t) \text{ in } L^{2}(\mathbb{T}^{d}) \quad \text{ as~ $\tau \to 0$, \quad for all~ $t \in (0, \infty)$}.
\end{equation*}
Here $\mu(t)$ is the solution of the effective Fokker-Planck equation~\eqref{effect} with initial value $\mu_{0}$.
\end{theorem}
\begin{theorem}\label{prop-homo-jko}
If $A_{\mathrm{hom}}$ is uniformly strongly convex, let $\overline{\mu}_{\tau}$ be the solution of the homogenized JKO scheme in Theorem~\ref{homo-of-jko} with initial value $\mu_{0}$. Then there exists $\overline{\mu}(t)$ such that $\overline{\mu}_{\tau}(t) \to \overline{\mu}(t)$ in $W_{2}$ metric. If we further have
\begin{equation*}
    \overline{\mu}_{\tau}(t) \to \overline{\mu}(t) \text{ in } W^{1, 1}(\mathbb{T}^{d}) \quad  \text{ as }~ \tau \to 0, \quad \text{ for all }~ t \in (0, \infty),
\end{equation*}
then $\overline{\mu}(t)$ is the solution to a quasilinear parabolic equation in divergence form with initial value $\mu_{0}$,
\begin{equation}\label{GH-limit}
    \frac{\partial \overline{\mu}}{\partial t} - \nabla \cdot \Big[\big(\tfrac{1}{2}\nabla A_{\mathrm{hom}}\big)^{-1}\big(\nabla \overline{\mu} + \overline{\mu}\nabla V(x)\big)\Big] = 0.
\end{equation}    
\end{theorem}

\begin{remark}
In general, as shown in Section~\ref{sec:homo-jko} below, one can prove that the $W^{1,1}$-Sobolev norm of the solution $\overline{\mu}_{\tau}(t)$ is uniformly bounded. Here, we further assume compactness of the time-discrete trajectories in the stronger space $W^{1,1}(\mathbb{T}^{d})$. Establishing such compactness is a technical issue that lies beyond the scope of this work. We note, however, that stronger convergence results for JKO schemes in the vanishing time-step limit have been established, for example, in \cite{lee2015jordan, santambrogio2024strong, coudreuse2025li} for standard Fokker--Planck equations on the torus.
\end{remark}

\subsection{Discussion and organization}\label{sec:dicuss}
From Theorems~\ref{prop-homo-l2} and~\ref{prop-homo-jko}, it is clear that the weighted $L^{2}$ scheme preserves the time-continuous homogenization limit~\eqref{effect} of the $\varepsilon$-Fokker--Planck equation~\eqref{0}. On the other hand, equation~\eqref{GH-limit} does not, in general, coincide with~\eqref{effect}. In particular, in one dimension, the computations in \cite{gao2023homogenization} give
\begin{equation*}
\big(\tfrac{1}{2}\nabla A_{\mathrm{hom}}\big)^{-1}(\xi)
=
\bigg[\int_{0}^{1}\frac{1}{A(x)^{1/2}}\ud x\bigg]^{-2}\xi,
\qquad
\text{whereas}
\qquad
A_{\ast}
=
\bigg[\int_{0}^{1}\frac{1}{A(x)}\ud x\bigg]^{-1}.
\end{equation*}
This shows that the $\varepsilon$-JKO scheme is not compatible with the time-continuous homogenization limit~\eqref{effect}. In summary, we draw the following diagram:
\[\begin{tikzcd}
	{\text{weighted } L^{2} \text{ scheme}} && {\text{homogenized scheme}} \\
	{\varepsilon\text{-Fokker Planck}} && {\text{effective limit}} \\
	{\varepsilon\text{-JKO scheme}} & {\text{homogenized JKO scheme}} & {\text{a limiting equation}}
	\arrow["{\varepsilon \to 0}", from=1-1, to=1-3]
	\arrow["{\tau \to 0}", from=1-1, to=2-1]
	\arrow["{\tau \to 0}", from=1-3, to=2-3]
	\arrow["{\varepsilon \to 0}", from=2-1, to=2-3]
	\arrow["{\tau \to 0}"', from=3-1, to=2-1]
	\arrow["{\varepsilon \to 0}", from=3-1, to=3-2]
	\arrow["{\tau \to 0}", from=3-2, to=3-3]
	\arrow["\neq"{description}, no head, from=3-3, to=2-3]
\end{tikzcd}\]
Throughout the proof of homogenization below, we observe that the effective Fokker--Planck equation~\eqref{effect} arises from a sequence of elliptic homogenization problems over the spatial domain at each time. In contrast, the homogenization of the $\varepsilon$-JKO scheme is based on the homogenization of particle trajectories. This is inconsistent with the time-continuous homogenization process, and consequently the corresponding homogenized solution differs from the effective limit~\eqref{effect}. Another possible explanation is the non commutativity between the limits $\varepsilon\to 0$ and $\tau\to 0$ in the $\varepsilon$-JKO scheme; we discuss this point further in Remark~\ref{remark-5}. We note that a similar non-commutativity result between spatial and temporal parameters was also observed in the convergence analysis of finite-volume approximations of Fokker--Planck equations \cite{forkert2022evolutionary, gladbach2020scaling}.

\begin{figure}[t!]
    \centering
    \includegraphics[width=\linewidth]{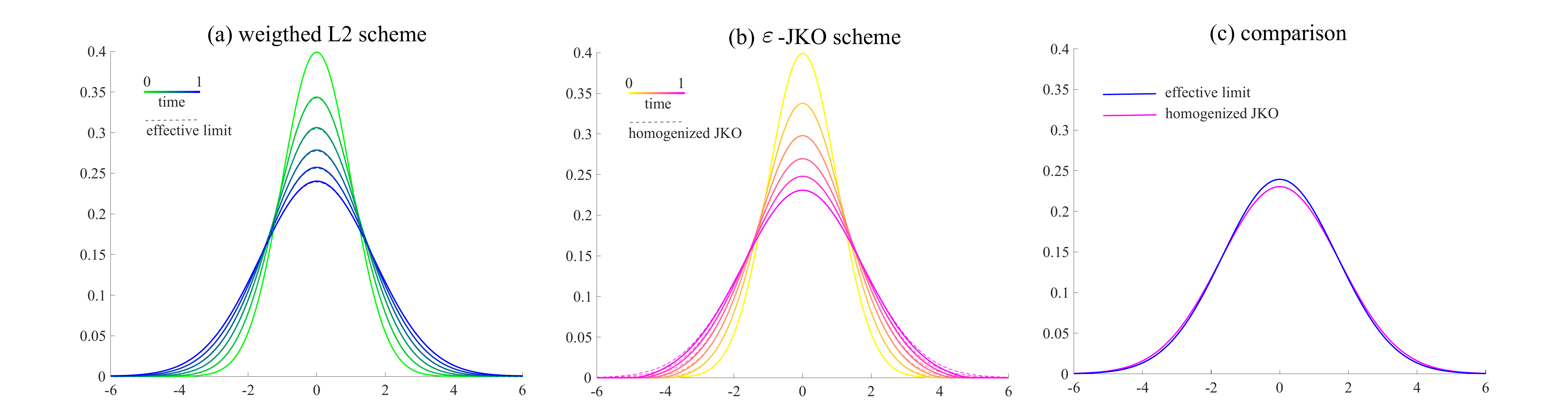}
    \caption{\rm Results of the time-discrete schemes with $A(x) = (1 + \frac{1}{2}\sin(2\pi x))^{2}$, $V_{\varepsilon}(x) = 0$ and $\varepsilon = 10^{-5}$. (a) Numerical solution of the weighted $L^{2}$ scheme, which closely matches the solution of the effective limit~\eqref{effect} illustrated by dash-line. (b) Numerical solution of the $\varepsilon$-JKO scheme, which closely matches the reference given by the solution of~\eqref{GH-limit} illustrated by dash-line. (c) Comparison between solutions of the homogenized weighted $L^{2}$ scheme and the homogenized JKO scheme at $t = 1$. For further details in the model and the implementation, we refer the readers to Appendix~\ref{appendix a}.}
    \label{fig:1}
\end{figure}

The rest of the paper is organized as follows. In Section~\ref{sec1}, we collect preliminary facts on optimal transport and $\Gamma$-convergence. In Sections~\ref{sec2} and~\ref{sec3}, we analyze the weighted $L^{2}$ scheme and the $\varepsilon$-JKO scheme, respectively. Finally, Section~\ref{sec4} contains our conclusions and discusses future directions.

\section{Preliminaries}\label{sec1}
\subsection{Optimal transport}\label{sec1-1}
In this subsection, we collect several preliminary results from optimal transport. For a thorough treatment of the subject, we refer the reader to the monographs \cite{ambrosio2008gradient, villani2008optimal}.

We mainly focus on the optimal transport on the flat tours $\mathbb{T}^{d}$, which is a compact manifold without boundary. Consequently, any probability measure $\mu \in P(\mathbb{T}^{d})$ has finite $p$ momentum for any $p \in \mathbb{N}$. Let $\mu, \nu \in P(\mathbb{T}^{d})$ be two probability measures and $\boldsymbol{\mathrm{r}} : \mathbb{T}^{d} \to \mathbb{T}^{d}$ be a measurable map. We say that $\boldsymbol{\mathrm{r}}$ push-forward (or transports) $\mu$ to $\nu$, denoted by $\nu = \boldsymbol{\mathrm{r}}\# \mu$, if $\nu[A] = \mu[\boldsymbol{\mathrm{r}}^{-1}(A)]$ for all Boreal sets $A$. The push-forward can be characterized using duality as follows:
\begin{equation*}
    \int_{\mathbb{T}^{d}} \varphi(\boldsymbol{\mathrm{r}}(x)) \ud \mu = \int_{\mathbb{T}^{d}} \varphi(x) \ud \nu, \text{ for all } \varphi \in C(\mathbb{T}^{d}).
\end{equation*}

Denote $W_{2}$ as the 2-Wasserstein distance:
\begin{equation}\label{W2}
    W_{2}(\mu, \nu) := \inf \bigg\{\int_{\mathbb{T}^{d} \times \mathbb{T}^{d}} d_{\mathrm{Eu}}(x, y)^{2} \ud \omega(x, y) : \omega \in \Pi(\mu, \nu)\bigg\}^{\frac{1}{2}},
\end{equation}
where $d_{\mathrm{Eu}}(x, y) = \min_{z \in \mathbb{Z}^{d}} |x - y + z|$ is the natural distance on flat torus $\mathbb{T}^{d}$, i.e., the distance given by the shortest path between two points on $\mathbb{T}^{d}$, and $\Pi(\mu, \nu)$ is the set of couplings between $\mu$ and $\nu$, i.e., the set of probability measures on $\mathbb{T}^{d} \times \mathbb{T}^{d}$, with marginal distribution $\mu$ and $\nu$ in each variable.

In general, let $c : \mathbb{T}^{d} \times \mathbb{T}^{d} \to \mathbb{R}$ be a squared distance, the Wasserstein distance associated with the cost $c(x, y)$ between two probability measures $\mu$ and $\nu$ is defined by the Monge-Kantrovich problem given as follows \cite[Chapter 6]{villani2008optimal},
\begin{equation*}
    W_{c}(\mu, \nu) = \inf \bigg\{\int_{\mathbb{T}^{d} \times \mathbb{T}^{d}} c(x, y) \ud \omega(x, y) : \omega \in \Pi(\mu, \nu)\bigg\}^{\frac{1}{2}}.
\end{equation*}
An alternative way of defining the Wasserstein distance is by minimizing the transport cost under the Monge problem,
\begin{equation*}
    W_{c}(\mu, \nu) = \inf\bigg\{\int_{\mathbb{T}^{d}} c(x, \boldsymbol{\mathrm{r}}(x)) \ud \mu : \boldsymbol{\mathrm{r}} \# \mu = \nu\bigg\}^{\frac{1}{2}}.
\end{equation*}

Let us recall the following solvability of Monge problems \cite[Theorem 10.28]{villani2008optimal}.
\begin{lemma}\label{monge}Let $c(x, y)$ be a squared distance on $\mathbb{T}^{d}$. Given absolutely continuous measures $\mu, \nu \in P(\mathbb{T}^{d})$, there exists a unique transport map $\boldsymbol{\mathrm{r}}_{\ast} : \mathbb{T}^{d} \to \mathbb{T}^{d}$ defined $\mu$ almost everywhere that solves the Monge problem
\begin{equation*}
    \min_{\boldsymbol{\mathrm{r}}} \int_{\mathbb{T}^{d}} c(x, \boldsymbol{\mathrm{r}}(x)) \ud \mu, \text{ s.t. } \boldsymbol{\mathrm{r}} \# \mu = \nu.
\end{equation*}
Furthermore, the optimal coupling $\omega_{\ast}$ in the corresponding Monge-Kantrovich problem is given by $\omega_{\ast} = (\mathrm{id}, \boldsymbol{\mathrm{r}}_{\ast}) \# \mu$.
\end{lemma}

Many cost functions $c(x,y)$ of interest arise from the least action principle
 (see, for example, \cite[Chapter 7]{villani2008optimal}):
\begin{equation*}
    c(x,y) = \min_{z}\bigg\{\int_0^1 L(z, \dot{z}) \ud t:\quad z(0) = x,~z(1)=y, \text{ and } z \in H^{1}\big([0, 1]; \mathbb{T}^{d}\big)\bigg\},
\end{equation*}
with a Lagrangian function $L(x, v)$ quadratic with respect to $v$. Minimizers of action functional $\int_0^1 L(z, \dot{z}) \ud t$ will be referred to as action minimizing curves. For smooth action function $L(x, v)$, the action minimizing curves $z(t)$ are given by the following Euler-Lagrange equation
\begin{equation*}
    \frac{\mathrm{d}}{\mathrm{d}t} (\nabla_{v}L)(z, \dot{z}) - \nabla_{x}L(z, \dot{z}) = 0.
\end{equation*}
In the following, we collect several facts in Riemannian geometry on action minimizing curves.
\begin{definition}
Let $L(x, v)$ be a Lagrangian action on $\mathbb{T}^{d}$, we define exponential map at $x$ as
\begin{equation*}
    \exp_{x} : \mathbb{R}^{d} \to \mathbb{T}^{d} \text{ such that } \xi \mapsto \exp_{x}(\xi) = z(1)\,,
\end{equation*}
where $z$ is the solution to the Euler-Lagrange equation with initial value $z(0) = x$ and $\dot{z}(0) = \xi$.
\end{definition}

On compact manifolds with Lagrangian action $L(x, v) = \langle A(x)v, v\rangle$, the exponential map $\exp_{x}$ is well-defined for all $\xi \in \mathbb{R}^{d}$ at every $x$, and any two given points can be connected via an action minimizing curve \cite[Chapter 1.5]{jost2005riemannian}.
\begin{lemma}[\cite{jost2005riemannian} Lemma 1.4.2]\label{velcoti-c}Given action function $L(x, v) = \langle A(x)v, v\rangle$ with matrix field $A$. Let $z$ be an action minimizing curve connecting $x$ to $y$, we have
\begin{equation*}
    v(t) := \la A(z(t))\dot{z}(t), \dot{z}(t) \ra \text{ remains constant.}
\end{equation*}
In particular, let $\xi_{x \to y} = \dot{z}(0)$ and $c(x, y)$ be the cost function given by the least action $\int_{0}^{1}L(z, \dot{z}) \ud t$, we can compute
\begin{equation*}
    c(x, y) = \la A(x)\xi_{x \to y}, \xi_{x \to y}\ra.
\end{equation*}
\end{lemma}

\begin{lemma}[\cite{jost2005riemannian} Equation (1.2.48)]\label{exp}Given smooth action function $L(x, v)$, and let $\exp$ be the corresponding exponential map, we have
\begin{equation*}
    \nabla \exp_{x}(0) = I_{d \times d}.
\end{equation*}
Hence, we have $\nabla (\varphi \circ \exp_{x})(0) = \nabla \varphi(x)$ for all $\varphi \in C^{1}(\mathbb{T}^{d})$.
\end{lemma}

We will frequently use the following equivalence of distances in our analysis (see, for example, \cite[Chapter 1.4]{jost2005riemannian}).
\begin{lemma}[equivalence between costs]\label{c-eqi}Let $L(x, v)$ be a Lagrangian action function on $\mathbb{T}^{d}$ such that $\Lambda_{0}|v|^{2} \le L(x, v) \le \Lambda_{1}|v|^{2}$, and $c(x, y)$ be the cost given by least action of $L(x, v)$, we have 
\begin{equation}\label{c-eqi-eq}
    \Lambda_{0}d_{\mathrm{Eu}}(x, y)^{2} \le c(x, y) \le \Lambda_{1}d_{\mathrm{Eu}}(x, y)^{2},
\end{equation}
where $d_{\mathrm{Eu}}(x, y) = \min_{z \in \mathbb{Z}^{d}} |x - y + z|$ is the natural distance on flat torus $\mathbb{T}^{d}$.
\end{lemma}
\begin{lemma}[equivalence between Wasserstein distances]\label{w-eqi}Let $c(x, y)$ be a cost satisfying~\eqref{c-eqi-eq} and let $W_{c}(\mu, \nu)$ be the Wasserstein distance associated with $c(x, y)$. For any $\mu$ and $\nu$, we have
\begin{equation*}
    \Lambda_{0}W_{2}(\mu, \nu)^{2} \le W_{c}(\mu, \nu)^{2} \le \Lambda_{1}W_{2}(\mu, \nu)^{2},
\end{equation*}
where $W_{2}(\mu, \nu)$ is the 2-Wasserstein distance defined in~\eqref{W2}.
\end{lemma}

Finally, we collect some discussions on the topology of $P(\mathbb{T}^{d})$.
\begin{definition}
We say a sequence of probability measures $(\mu_{\ell})$ narrowly converges to a measure $\mu$, if
\begin{equation}\label{narrow-convergence}
    \lim_{\ell \to \infty} \int_{\mathbb{T}^{d}} \varphi(x) \ud \mu_{\ell} = \int_{\mathbb{T}^{d}} \varphi(x) \ud \mu, \text{ for all } \varphi \in C(\mathbb{T}^{d}).
\end{equation}
\end{definition}

We note that the commonly used definition for narrow convergence is in duality with the space of bounded continuous functions $C_{b}(\mathbb{T}^{d})$. But, since $\mathbb{T}^{d}$ is compact, $C_{b}(\mathbb{T}^{d}) = C(\mathbb{T}^{d})$.

\begin{definition}We say that a sequence of probability measures $(\mu_{\ell})$ converges to $\mu$ in $W_{2}$ if
\begin{equation*}
    \lim_{\ell \to \infty} W_{2}(\mu_{\ell}, \mu) = 0.
\end{equation*}
\end{definition}

\begin{lemma}[Definition 6.8 (iii) in \cite{villani2008optimal} and Theorem 6.9 in \cite{villani2008optimal}]\label{equvial-W-N}$(\mu_{\ell})$ converges to $\mu$ in $W_{2}$ if and only if the following is satisfied for some $x_{0} \in \mathbb{T}^{d}$:
\begin{equation*}
    \mu_{\ell} \to \mu \text{ narrowly, and } \lim_{R \to \infty} \limsup_{\ell \to \infty} \int_{d_{\mathrm{Eu}}(x, x_{0}) \ge R} d_{\mathrm{Eu}}(x, x_{0})^{2} \ud \mu_{\ell} = 0.
\end{equation*}
Therefore, on a compact domain $\mathbb{T}^{d}$, the metric topology is equivalent to the narrowly topology since $\sup_{x \in \mathbb{T}^{d}}d_{\mathrm{Eu}}(x, x_{0}) < \infty$.
\end{lemma}

\begin{lemma}[Remark 6.19 in \cite{villani2008optimal}]\label{comp-1}The metric space $(P(\mathbb{T}^{d}), W_{2})$ is compact.
\end{lemma}
\begin{proof}
Since $\mathbb{T}^{d}$ is compact, any sequence of probability measures is tight. Therefore, by Prokhorov's theorem (see, for example, \cite[page 43]{villani2008optimal}), $P(\mathbb{T}^{d})$ is compact with respect to the narrow topology. As the result of the equivalence in Lemma~\ref{equvial-W-N}, the metric space $(P(\mathbb{T}^{d}), W_{2})$ is compact.
\end{proof}

\subsection{$\Gamma$-convergence}
In this subsection, we briefly recall the theory of the $\Gamma$-convergence. For detailed discussions on the subject of $\Gamma$-convergence, we refer the reader to the monographs \cite{braides2002gamma, dal2012introduction}.
\begin{definition}[Definition 1.5 in \cite{braides2002gamma} and Definition 4.1 in \cite{dal2012introduction}]Given a sequence of functionals $\{F_{\varepsilon}\}$ on metric space $(X, d)$, we say $F_{\varepsilon}$ $\Gamma$-converges to a functional $F$, if for all $u \in X$, the following two conditions holds true
\begin{enumerate}
    \item \textnormal{($\liminf$ inequality)} for all sequence $\{u_{\varepsilon}\}$ converging to $u$ in $X$,
    \begin{equation*}
        \liminf_{\varepsilon \to 0} F_{\varepsilon}(u_{\varepsilon}) \ge F(u).
    \end{equation*}
    \item \textnormal{(existence of a recovery sequence)} there exists a sequence $\{u_{\varepsilon}\}$ converging to $u$ in $X$ such that
    \begin{equation*}
        \limsup_{\varepsilon \to 0} F_{\varepsilon}(u_{\varepsilon}) \le F(u).
    \end{equation*}
\end{enumerate}
The functional $F$ is the $\Gamma$-limit of $(F_{\varepsilon})$, and we write $F = \Gamma\text{-}\lim_{\varepsilon \to 0} F_{\varepsilon}$.
\end{definition}

We recall the following notion of continuous convergence.
\begin{definition}[Proposition 2.3 in \cite{braides2014local}]Given a sequence of continuous functionals $\{G_{\varepsilon}\}$ on metric space $(X, d)$, we say $G_{\varepsilon}$ continuously converges to $G$, if for all $u \in X$ and for all $u_{\varepsilon} \to u$ in $X$, we have $G_{\varepsilon}(u_{\varepsilon}) \to G(u)$.
\end{definition}
An important property of the $\Gamma$-convergence is its stability under continuous perturbations \cite[Proposition 2.3]{braides2014local}.
\begin{lemma}\label{stability-Gamma-thm}
Let $\{F_{\varepsilon}\}$ and $\{G_{\varepsilon}\}$ be a sequence of functionals and continuous functionals, respectively. If $F_{\varepsilon}$ $\Gamma$-converges to $F$, and $G_{\varepsilon}$ continuously converges to $G$ in the metric $d$, we have
\begin{equation}\label{stability-Gamma}
    \Gamma\text{-}\lim_{\varepsilon \to 0} (F_{\varepsilon} + G_{\varepsilon}) = F + G.
\end{equation}
\end{lemma}

We recall the following Urysohn property of $\Gamma$-convergence \cite[Proposition 1.44]{braides2002gamma}.
\begin{lemma}\label{Urysohn} Let $\{F_{\varepsilon}\}$ be a sequence of functionals, the following two are equivalent
\begin{enumerate}
    \item $\Gamma\textnormal{-}\lim_{\varepsilon \to 0} F_{\varepsilon} = F$.
    \item for every subsequence $\{\varepsilon_{\ell}\}$ 
    there exists a further subsequence $\{\varepsilon_{\ell(\nu)}\}$, such that $\Gamma\textnormal{-}\lim_{\nu \to \infty} F_{\varepsilon_{\ell(\nu)}} = F$.
\end{enumerate}
\end{lemma}

We recall the following equi-coercive condition which implies the compactness of a sequence of minimizers $(u_{\varepsilon})$.
\begin{definition}[coercivity]We say a sequence of functionals $\{F_{\varepsilon}\}$ on metric space $(X, d)$ is equi-coercive, if there exists a non-empty compact subset $K$ such that $\inf_{X} F_{\varepsilon} = \inf_{K} F_{\varepsilon}$ for all $\varepsilon > 0$.
\end{definition}

We also recall the following fundamental theorem of $\Gamma$-convergence (see \cite[Theorem 1.21]{braides2002gamma} and \cite[Theorem 2.1]{braides2014local}).
\begin{theorem}[fundamental theorem]\label{fund-G}Let $\{F_{\varepsilon}\}$ be a sequence of equi-coercive functionals on $(X, d)$ which $\Gamma$-converges $F$, we have
\begin{enumerate}
    \item F admits a minimum, and  $\min F = \lim_{\varepsilon \to 0} \min F_{\varepsilon}$.
    \item Let $u_{\varepsilon}$ be a minimizer of $F_\varepsilon$ for any $\varepsilon$, then there exists a subsequence $\{u_{\varepsilon_\ell}\}$
    such that $u_{\varepsilon_{\ell}} \to u_{\ast}$ in $X$ as $\ell \to \infty$, where $u_\ast$ is a minimizer of $F(u)$. If, in addition, the minimizer of $F$ is unique, we have $u_{\varepsilon} \to u_{\ast}$ as $\varepsilon \to 0$.
\end{enumerate}
\end{theorem}

\section{Weighted $L^{2}$ scheme}\label{sec2}
In this section, we analyze the weighted $L^{2}$ scheme constructed in Definition~\ref{L2-scheme-def}, where the solutions are defined according to,
\begin{equation}\label{k}
    \min_{H^{1}(\mathbb{T}^{d})} \mathcal{I}_{\varepsilon, \tau}(\mu, \mu^{n - 1}_{\varepsilon}) := Q_{\varepsilon}(\mu) + \frac{1}{2\tau} K_{\varepsilon}(\mu, \mu_{\varepsilon}^{n - 1})^{2},
\end{equation}
where $Q_{\varepsilon}$ and $K_{\varepsilon}$ are the energy and distance defined in~\eqref{disK}.

Based on ellipticity of $A$ and regularities of $V_{\varepsilon}$ from Assumptions~\ref{A3},~\ref{A5},~and~\ref{A4}, the existence of a unique minimizer to~\eqref{k} is guaranteed by standard direct methods in the calculus of variations \cite[Chapter 1]{struwe2000variational} and \cite[Chapter 3]{dacorogna2008direct}. Therefore, we have the well-posedness of the weighted $L^{2}$ scheme as follows.
\begin{proposition}[well-posedness]\label{WL2-existence}Given $\varepsilon, \tau > 0$ and $\mu_{0} \in L^{2}(\mathbb{T}^{d})$, there exists a unique solution $\mu_{\varepsilon, \tau}$ to the weighted $L^{2}$ scheme in Definition~\ref{L2-scheme-def} with initial value $\mu_{0}$.
\end{proposition}

By taking Gateaux derivative of~\eqref{k} in $H^{1}(\mathbb{T}^{d})$ at the minimizer $\mu_{\varepsilon}^{n}$, one can check that $(\mu_{\varepsilon}^{n})_{n \in \mathbb{N}}$ obeys the backward Euler scheme of~\eqref{0} in the $H^{1}$ weak sense
\begin{equation}\label{L2-scheme}
    \begin{aligned}
        \frac{\mu^{n}_{\varepsilon} - \mu^{n - 1}_{\varepsilon}}{\tau} &- \nabla \cdot \Big[A\Big(\frac{x}{\varepsilon}\Big) \nabla \mu^{n}_{\varepsilon} + \mu^{n}_{\varepsilon}A\Big(\frac{x}{\varepsilon}\Big)\nabla V_{\varepsilon}(x)\Big] = 0.
    \end{aligned}
\end{equation}
As the convergence of the backward Euler scheme~\eqref{L2-scheme} to the linear parabolic equation~\eqref{0} is well established in numerical analysis \cite{marcellini1978periodic}, we state the following result without proof.
\begin{proposition}[convergence]\label{IE-time}Given $\varepsilon, \tau > 0$ and initial value $\mu_{0} \in L^{2}(\mathbb{T}^{d})$. Let $\mu_{\varepsilon, \tau}$ be the solution of the weighted $L^{2}$ scheme in Definition~\ref{L2-scheme-def}, we have
\begin{equation*}
    \mu_{\varepsilon, \tau}(t) \to \mu_{\varepsilon}(t) \text{ in } L^{2}(\mathbb{T}^{d}) \text{ as } \tau \to 0,
\end{equation*}
where $\mu_{\varepsilon}(t)$ is the solution of~\eqref{0} with initial value $\mu_{0}$.
\end{proposition}

We note that Theorem~\ref{prop-homo-l2} can be proved using the same argument as the present proposition. We therefore focus on the proof of Theorem~\ref{homo-of-ie} in this section, which consists of two main steps: establishing the $\Gamma$-convergence (Section~\ref{sec31}) and proving the required compactness (Section~\ref{sec:2-2}). These results then allow us to invoke the fundamental theorem of $\Gamma$-convergence, namely, Theorem~\ref{fund-G}.

\subsection{$\Gamma$-convergence} \label{sec31}
In this subsection, we analyze the homogenization of the weighted $L^{2}$ scheme in Theorem~\ref{homo-of-ie} with a fixed $\tau$. We first establish the following $\Gamma$-convergence result.  
\begin{theorem}\label{Gamma-i-lemma}
If $\mu^{n - 1}_{\varepsilon} \to \mu^{n - 1}$ in $L^{2}(\mathbb{T}^{d})$, we have the following $\Gamma$-convergence (as functional over $\mu$) in the $L^{2}$ metric,
\begin{equation}\label{Gamma-i}
    \Gamma\text{-}\lim_{\varepsilon \to 0} \mathcal{I}_{\varepsilon, \tau}(\mu, \mu_{\varepsilon}^{n - 1}) = \mathcal{I}_{\ast, \tau}(\mu, \mu^{n - 1}) := Q(\mu) + \frac{1}{2\tau}K(\mu, \mu^{n - 1})^{2}
\end{equation}
where $Q$ and $K$ are homogenized energy and distance defined in~\eqref{limit-I}.
\end{theorem}

Due to the ellipticity of $A_{\ast}$ in~\eqref{B-hom-1} and regularities of $V$ in Assumptions~\ref{A4} and~\ref{A6}, the existence of a unique minimizer is guaranteed by direct methods.
\begin{proposition}[well-posedness]\label{exist-limit-u}Given $\varepsilon, \tau > 0$ and $\mu^{n - 1} \in L^{2}(\mathbb{T}^{d})$, there exists a unique solution $\mu^{n} \in H^{1}(\mathbb{T}^{d})$ to the $\Gamma$-limit $\min_{\mu}\mathcal{I}_{\ast, \tau}(\mu, \mu^{n - 1})$ in~\eqref{Gamma-i}. Hence, the homogenized weighted $L^{2}$ scheme in Theorem~\ref{homo-of-ie} is well-posed.
\end{proposition}

To prove the $\Gamma$-convergence in Theorem~\ref{Gamma-i-lemma}, it is sufficient to establish the following two lemmas, by virtue of the stability of $\Gamma$-convergence under continuous perturbations stated in Lemma~\ref{stability-Gamma-thm}. The proofs of these lemmas are provided in the next two subsections.
\begin{lemma}\label{E-gamma}
In the $L^{2}$ metric, we have
\begin{equation*}
    \Gamma\text{-}\lim_{\varepsilon \to 0} Q_{\varepsilon} = Q.
\end{equation*}
\end{lemma}
\begin{lemma}\label{c-converge-K}
The distances $K_{\varepsilon}$ and $K$ obey the following,
\begin{equation*}
    \lim_{\varepsilon \to 0}K_{\varepsilon}(\mu_{\varepsilon}, \mu_{\varepsilon}^{n - 1})^{2} = K(\mu, \mu^{n - 1})^{2}
\end{equation*}
for all sequences $\{\mu_{\varepsilon}\}$ and $\{\mu_{\varepsilon}^{n - 1}\}$ such that $\mu_{\varepsilon} \to \mu$ and $\mu^{n - 1}_{\varepsilon} \to \mu^{n - 1}$ in $L^{2}(\mathbb{T}^{d})$.
\end{lemma}

\subsubsection{Homogenization of energy}
We prove the $\Gamma$-convergence in Lemma~\ref{E-gamma} by regarding the energies as functionals on $\varrho_\varepsilon := \mu/\pi_{\varepsilon}$ and $\varrho_{\ast} := \mu/\pi_{\ast}$ via a change of variable:
\begin{align*}
    \widehat{Q}_{\varepsilon}(\varrho_\varepsilon) &:= \frac{1}{2}\int_{\mathbb{T}^{d}}\Big<\pi_{\varepsilon}(x) A\Big(\frac{x}{\varepsilon}\Big)\nabla \varrho_\varepsilon, \,\nabla \varrho_\varepsilon\Big> \ud x,\\
    \widehat{Q}(\varrho_{\ast}) &:= \frac{1}{2}\int_{\mathbb{T}^{d}}\big<\pi_{\ast}(x) A_{\ast}\nabla \varrho_{\ast}, \nabla \varrho_{\ast}\big> \ud x.
\end{align*}

Instead of directly proving Lemma~\ref{E-gamma}, we prove the following lemma and obtain Lemma~\ref{E-gamma} as a corollary. 
\begin{lemma}\label{lemma34}
For any $\varrho \in H^{1}(\mathbb{T}^{d})$, we have the following $\Gamma$-convergence,
\begin{equation}\label{E-gamma-w}
    \Gamma\text{-}\lim_{\varepsilon \to 0} \widehat{Q}_{\varepsilon}(\varrho) = \widehat{Q}(\varrho),
\end{equation}
in the $L^{2}$ metric.
\end{lemma}

\begin{proof}[Proof of Lemma~\ref{E-gamma}]
Let $(\mu_{\varepsilon})$ be any sequence converging to $\mu$ in $L^{2}(\mathbb{T}^{d})$, by the uniform convergence of $V_{\varepsilon} \to V$ in Assumption~\ref{A6}, we have $\pi_{\varepsilon}^{-1} \to \pi_{\ast}^{-1}$ uniformly, and therefore,
\begin{equation*}
    \mu_{\varepsilon} \to \mu \text{ in $L^{2}(\mathbb{T}^{d})$ } \Longrightarrow \varrho_{\varepsilon} = \mu_{\varepsilon}/\pi_{\varepsilon} \to \varrho_{\ast} = \mu/\pi_{\ast} \text{ in $L^{2}(\mathbb{T}^{d})$}\,.
\end{equation*}
Combining the $\liminf$-inequality from the $\Gamma$-convergence in~\eqref{E-gamma-w}, we have
\begin{equation*}
    \liminf_{\varepsilon \to 0} Q_{\varepsilon}(\mu_{\varepsilon}) = \liminf_{\varepsilon \to 0} \widehat{Q}_{\varepsilon}(\varrho_{\varepsilon}) \ge \widehat{Q}(\varrho_{\ast}) = Q(\mu).
\end{equation*}
On the other hand, given any $\mu$, we let $\varrho_{\ast} = \mu/\pi_{\ast}$, and let $\{\varrho_{\varepsilon}\}$ be a recovery sequence for $\varrho_{\ast}$ according to the $\Gamma$-convergence of $\widehat{Q}_{\varepsilon}(\varrho)$. Let $ \mu_{\varepsilon} = \varrho_{\varepsilon} \pi_{\varepsilon}$, then we have
\begin{equation*}
    \limsup_{\varepsilon \to 0} Q_{\varepsilon}(\mu_{\varepsilon}) = \limsup_{\varepsilon \to 0} \widehat{Q}_{\varepsilon}(\varrho_{\varepsilon}) \le \widehat{Q}(\varrho) = Q(\mu).
\end{equation*}
\end{proof}

Now, we prove~\eqref{E-gamma-w} as the generalization of the following classical elliptic homogenization \cite{marcellini1978periodic, pavliotis2008multiscale, bensoussan2011asymptotic}, and \cite[Chatper 4]{alouges2016introduction}, by incorporating the inhomogeneity given by $\pi_{\varepsilon}$ and $\pi_{\ast}$ according to \cite[Section 6]{weinan1991class}.
\begin{lemma}[elliptic homogenization]\label{eh-lemma}Let
\begin{equation*}
    \begin{aligned}
        F_{\varepsilon}(\varrho) = \int_{\mathbb{T}^{d}}\Big<A\Big(\frac{x}{\varepsilon}\Big)\nabla \varrho, \nabla \varrho\Big> \ud x \text{ and } F(\varrho) = \int_{\mathbb{T}^{d}}\big<A_{\ast}\nabla \varrho, \nabla \varrho\big> \ud x,
    \end{aligned}
\end{equation*}
where $A_{\ast}$ is the effective diffusion coefficient defined in~\eqref{B-hom}, we have $\Gamma\text{-}\lim_{\varepsilon \to 0} F_{\varepsilon} = F$ in the $L^{2}$ metric.
\end{lemma}

\begin{proof}[Proof of Lemma~\ref{lemma34}]Given $H^{1}(\mathbb{T}^{d})$ function $\varrho$ and measurable set $E \subseteq \mathbb{T}^{d}$, we let
\begin{equation*}
    \widehat{Q}_{\varepsilon}(\varrho, E) = \frac{1}{2}\int_{E} \Big<\pi_{\varepsilon}(x)A\Big(\frac{x}{\varepsilon}\Big)\nabla \varrho, \nabla \varrho\Big> \ud x.
\end{equation*}
To prove the convergence in the homogenization limit $\varepsilon \to 0$, we use the Urysohn property of $\Gamma$-convergence stated in Lemma~\ref{Urysohn}. Specifically, we show that for every subsequence $\{\varepsilon_{\ell}\}$ satisfying $\varepsilon_{\ell} \to 0$ as $\ell \to \infty$, there exists a further subsequence ${\varepsilon_{\ell(\nu)}}$ such that $\Gamma\text{-}\lim_{\nu \to \infty} \widehat{Q}_{\varepsilon_{\ell(\nu)}} \to \widehat{Q}$.

To this end, we first invoke the compactness of functionals $\widehat{Q}_{\varepsilon_{\ell}}$ under $\Gamma$-converges (see, for example, \cite[Section 3]{weinan1991class}) to get
\begin{equation}\label{gamma-loc-h}
    \Gamma\text{-}\lim_{\nu \to \infty} \widehat{Q}_{\varepsilon_{\ell(\nu)}}(\varrho, E) = \frac{1}{2}\int_{E} h(x) \ud x, \text{ for all } E \in \mathscr{D},
\end{equation}
for a further subsequence $\{\varepsilon_{\ell(\nu)}\}$, where $\mathscr{D}$ is the algebra generated by all open cubes with rational vertices, and $h \in L^{1}_{\text{loc}}(\mathbb{R}^{d})$ is an integrant to be identified in the following.

Let $x_{0}$ be a Lebesgue point of $h$, and
\begin{equation*}
    \widehat{Q}_{\varepsilon}^{x_{0}}(\varrho, E) := \frac{1}{2} \int_{E} \Big<\pi_{\varepsilon}(x_{0}) A\Big(\frac{x}{\varepsilon}\Big)\nabla \varrho, \nabla \varrho\Big> \ud x,
\end{equation*}
be the functional localized at $x_{0}$. According to Lemma~\ref{eh-lemma} and the uniform convergence $\pi_{\varepsilon} \to \pi_{\ast}$, we can show that
\begin{equation}\label{E-homo}
    \Gamma\text{-}\lim_{\varepsilon \to 0} \widehat{Q}_{\varepsilon}^{x_{0}}(\varrho, E) = \widehat{Q}^{x_{0}}(\varrho, E) := \frac{1}{2}\int_{E} \big<\pi_{\ast}(x_{0})A_{\ast}\nabla \varrho, \nabla \varrho\big> \ud x.
\end{equation}
Since $\pi_{\varepsilon} = e^{-V_{\varepsilon}}$ is equi-Lipschitz continuous by Assumption~\ref{A4}, there exists a constant $C$ such that the error of localization can be controlled by,
\begin{equation}\label{E1-3}
    \begin{aligned}
        \big|\widehat{Q}_{\varepsilon}(\varrho, E) - \widehat{Q}_{\varepsilon}^{x_{0}}(\varrho, E)\big| &\le \frac{1}{2}\int_{E} \big|\pi_{\varepsilon}(x) - \pi_{\varepsilon}(x_{0})\big| \Big<A\Big(\frac{x}{\varepsilon}\Big)\nabla \varrho, \nabla \varrho\Big> \ud x\\
        &\le C\int_{E} |x - x_{0}| \times |\nabla \varrho|^{2} \ud x\\
        &\le C \operatorname{diam}(E)\int_{E} |\nabla \varrho|^{2} \ud x,
    \end{aligned}
\end{equation}
where $\operatorname{diam}(E)$ is the diameter of the set $E$.
\vspace{6pt}

\noindent
i) (upper bound) By the $\Gamma$-convergence of the localized functional $\widehat{Q}_{\varepsilon}^{x_{0}}(\varrho)$ in~\eqref{E-homo}, we can find a recovery sequence $(\varrho_{\varepsilon_{\ell(\nu)}}) \subseteq H^{1}(\mathbb{T}^{d})$ for $\varrho$ such that $\varrho_{\varepsilon_{\ell(\nu)}} \to \varrho$ in $L^{2}(\mathbb{T}^{d})$ with
\begin{equation*}
    \lim_{\nu \to \infty} \widehat{Q}_{\varepsilon_{\ell(\nu)}}^{x_{0}}(\varrho_{\varepsilon_{\ell(\nu)}}, E) = \widehat{Q}^{x_{0}}(\varrho, E) \text{ for } E \in \mathscr{D}.
\end{equation*}
According to the $\liminf$-inequality from~\eqref{gamma-loc-h}, we have
\begin{equation*}
    \liminf_{\nu \to \infty} \widehat{Q}_{\varepsilon_{\ell(\nu)}}(\varrho_{\varepsilon_{\ell(\nu)}}, E) \ge \frac{1}{2}\int_{E} h(x) \ud x.
\end{equation*}
Hence, by taking the above two limits into~\eqref{E1-3}, we can get
\begin{equation}\label{Qub}
    \begin{aligned}
        \frac{1}{2}\int_{E} h(x) \ud x &\le \liminf_{\nu \to \infty} \widehat{Q}_{\varepsilon_{\ell(\nu)}}(\varrho_{\varepsilon_{\ell(\nu)}}, E)\\
        &\le \liminf_{\nu \to \infty} \widehat{Q}_{\varepsilon_{\ell(\nu)}}^{x_{0}}(\varrho_{\varepsilon_{\ell(\nu)}}, E) + C \operatorname{diam}(E)\int_{E} |\nabla \varrho_{\varepsilon_{\ell(\nu)}}|^{2} \ud x\\
        &= \lim_{\nu \to \infty} \widehat{Q}_{\varepsilon_{\ell(\nu)}}^{x_{0}}(\varrho_{\varepsilon_{\ell(\nu)}}, E) + \liminf_{\nu \to \infty} C \operatorname{diam}(E)\int_{E} |\nabla \varrho_{\varepsilon_{\ell(\nu)}}|^{2} \ud x\\
        &\le \widehat{Q}^{x_{0}}(\varrho, E) + C_{0} \operatorname{diam}(E) \times |E|
    \end{aligned}
\end{equation}
for a constant $C_{0}$.
\vspace{6pt}

\noindent
ii) (lower bound) On the other hand, by the definition of $\Gamma$-convergence, we can find a recovery sequence $(\varrho_{\varepsilon_{\ell(\nu)}}) \subseteq H^{1}(\mathbb{T}^{d})$ for $\varrho$ from~\eqref{gamma-loc-h}, such that $\varrho_{\varepsilon_{\ell(\nu)}} \to \varrho$ in $L^{2}(\mathbb{T}^{d})$ with 
\begin{equation*}
    \lim_{\nu \to \infty} \widehat{Q}_{\varepsilon_{\ell(\nu)}}(\varrho_{\varepsilon_{\ell(\nu)}}, E) = \frac{1}{2}\int_{E} h(x) \ud x \text{ for } E \in \mathscr{D}.
\end{equation*}
In the following, we derive a lower bound for $\frac{1}{2}\int_{E} h(x) \ud x$ using $\widehat{Q}^{x_{0}}(\varrho)$.

To simply the notation, we first fixed a scale $\varepsilon := \varepsilon_{\ell(\nu)}$. 

Given a positive real number $M$, we denote (where $u \wedge v := \min(u, v)$ and $u \vee v := \max(u, v)$)
\begin{equation*}
    \varrho_{\varepsilon, M} := \big(\varrho_{\varepsilon} \wedge (\varrho + 1/M)\big) \vee (\varrho - 1/M),
\end{equation*}
and
\begin{equation*}
    E_{M, \varepsilon} := \big\{x \in E : |\varrho_{\varepsilon}(x) - \varrho(x)| > 1/M\big\}.
\end{equation*}
To estimate the error introduced by truncation, we split the integral,
\begin{equation}\label{E-0}
    \begin{aligned}
        \widehat{Q}_{\varepsilon}(\varrho_{\varepsilon, M}, E) &= \frac{1}{2} \int_{E} \Big<\pi_{\varepsilon}(x) A\Big(\frac{x}{\varepsilon}\Big)\nabla \varrho_{\varepsilon, M}, \nabla \varrho_{\varepsilon, M}\Big> \ud x\\
        &= \widehat{Q}_{\varepsilon}(\varrho_{\varepsilon}, E - E_{M, \varepsilon}) + \frac{1}{2}\int_{E_{M, \varepsilon}} \Big<\pi_{\varepsilon}(x) A\Big(\frac{x}{\varepsilon}\Big)\nabla \varrho_{\varepsilon, M}, \nabla \varrho_{\varepsilon, M}\Big> \ud x
    \end{aligned}
\end{equation}
By the definition of $E_{M, \varepsilon}$, we have $\varrho_{\varepsilon, M} = \varrho_{\varepsilon}$ in $E - E_{M, \varepsilon}$. Therefore,
\begin{equation}\label{E-1}
    \begin{aligned}
        \widehat{Q}_{\varepsilon}(\varrho_{\varepsilon}, E - E_{M, \varepsilon}) = \frac{1}{2} \int_{E - E_{M, \varepsilon}} \Big<\pi_{\varepsilon}(x) A\Big(\frac{x}{\varepsilon}\Big)\nabla \varrho_{\varepsilon, M}, \nabla \varrho_{\varepsilon, M}\Big> \ud x.
    \end{aligned}
\end{equation}
Furthermore, since $A(x) \le \Lambda_{1} I_{d \times d}$ and $\pi_{\varepsilon}(x) \le 1$ by Assumptions~\ref{A3} and~\ref{A5}, we have
\begin{equation}\label{E-2}
    \begin{aligned}
        \int_{E_{M, \varepsilon}} \Big<\pi_{\varepsilon}(x) A\Big(\frac{x}{\varepsilon}\Big)\nabla \varrho_{\varepsilon, M},\, &\nabla \varrho_{\varepsilon, M}\Big> \ud x\\
        &\le \Lambda_{1}\int_{E_{M, \varepsilon}} |\nabla \varrho_{\varepsilon, M}|^{2} \ud x = \Lambda_{1} \int_{E_{M, \varepsilon}} |\nabla \varrho|^{2} \ud x
    \end{aligned}
\end{equation}
Hence, by taking~\eqref{E-1} and~\eqref{E-2} into the integral~\eqref{E-0}, we can get
\begin{equation}\label{E1-1}
    \widehat{Q}_{\varepsilon}(\varrho_{\varepsilon, M}, E) \le \widehat{Q}_{\varepsilon}(\varrho_{\varepsilon}, E - E_{M, \varepsilon}) + \frac{\Lambda_{1}}{2} \int_{E_{M, \varepsilon}} |\nabla \varrho|^{2} \ud x.
\end{equation}

Now, taking~\eqref{E1-1} into the localization error in~\eqref{E1-3}, we can show that,
\begin{equation*}
    \begin{aligned}
        \widehat{Q}_{\varepsilon}(\varrho_{\varepsilon}, E - E_{M, \varepsilon}) &\ge \widehat{Q}_{\varepsilon}^{x_{0}}(\varrho_{\varepsilon, M}, E)\\
        &\quad - \frac{\Lambda_{1}}{2} \int_{E_{M, \varepsilon}} |\nabla \varrho|^{2} \ud x - C \text{diam}(E)\int_{E} |\nabla \varrho_{\varepsilon, M}|^{2} \ud x.
    \end{aligned}
\end{equation*}
Hence, taking $\varepsilon = \varepsilon_{\ell(\nu)}$ in the above inequality, we can get
\begin{equation}\label{E-5}
    \begin{aligned}
        \widehat{Q}_{\varepsilon_{\ell(\nu)}}(\varrho_{\varepsilon_{\ell(\nu)}}, E - E_{M, \varepsilon_{\ell(\nu)}}) &\ge \widehat{Q}_{\varepsilon_{\ell(\nu)}}^{x_{0}}(\varrho_{\varepsilon_{\ell(\nu)}, M}, E)\\
        &\quad - \frac{\Lambda_{1}}{2} \int_{E_{M, \varepsilon_{\ell(\nu)}}} |\nabla \varrho|^{2} \ud x - C_{1} \text{diam}(E) \times |E|
    \end{aligned}
\end{equation}
for a constant $C_{1}$. By its definition, the convergence $\varrho_{\varepsilon_{\ell(\nu)}} \to \varrho$ in $L^{2}$ implies that
\begin{equation*}
    |E_{M, \varepsilon_{\ell(\nu)}}| \le M^{2}\int_{E_{M, \varepsilon_{\ell(\nu)}}} |\varrho_{\varepsilon_{\ell(\nu)}} - \varrho|^{2} \ud x \le M^{2}\int_{\mathbb{T}^{d}} |\varrho_{\varepsilon_{\ell(\nu)}} - \varrho|^{2} \ud x \to 0 \text{ as $\nu \to \infty$.}
\end{equation*}
Hence, if we let $\nu \to \infty$ and then let $M \to \infty$ in~\eqref{E-5}, we get
\begin{equation}\label{Elb}
    \begin{aligned}
        \frac{1}{2}\int_{E} h(x) \ud x &= \lim_{M \to \infty}\lim_{\nu \to \infty}\widehat{Q}_{\varepsilon_{\ell(\nu)}}(\varrho_{\varepsilon_{\ell(\nu)}}, E - E_{M, \varepsilon_{\ell(\nu)}})\\
        &= \liminf_{M \to \infty}\liminf_{\nu \to \infty}\widehat{Q}_{\varepsilon_{\ell(\nu)}}(\varrho_{\varepsilon_{\ell(\nu)}}, E - E_{M, \varepsilon_{\ell(\nu)}})\\
        &\ge \widehat{Q}^{x_{0}}(\varrho, E) - C_{1} \text{diam}(E) \times |E|,
    \end{aligned}
\end{equation}
where the last inequality is given by the $\liminf$-inequality from the $\Gamma$-convergence of the localized functional in~\eqref{E-homo}.

\vspace{6pt}

\noindent
iii) Finally, by combining~\eqref{Qub} and~\eqref{Elb}, there exists a constant $C$ such that
\begin{equation*}
    \frac{\big|\frac{1}{2}\int_{E} h(x) \ud x - \widehat{Q}^{x_{0}}(\varrho, E)\big|}{|E|} \le C \text{diam}(E),
\end{equation*}
Let $E = B_{R}(x_{0})$, and then let $R \to 0$, we have
\begin{equation*}
    \begin{aligned}
        h(x_{0}) :=& \lim_{R \to 0}\frac{1}{|B_{R}(x_{0})|} \widehat{Q}(\varrho, B_{R}(x_{0}))\\
        =& \lim_{R \to 0}\frac{1}{|B_{R}(x_{0})|} \widehat{Q}^{x_{0}}(\varrho, B_{R}(x_{0})) = \big<\pi_{\ast}(x_{0})A_{\ast}\nabla \varrho(x_{0}), \nabla \varrho(x_{0})\big>.
    \end{aligned}
\end{equation*}
Since $h$ is independent of the choice of a subsequence $(\varepsilon_{\ell(\nu)})$, we have, from~\eqref{gamma-loc-h},
\begin{equation*}
    \Gamma\text{-}\lim_{\varepsilon \to 0} \widehat{Q}_{\varepsilon}(\varrho, E) = \frac{1}{2}\int_{E} h(x) \ud x = \widehat{Q}(\varrho, E),
\end{equation*}
which proves the claim~\eqref{E-gamma-w}.
\end{proof}

\subsubsection{Convergence of distance}
\begin{proof}[Proof of Lemma~\ref{c-converge-K}]
By using Assumption~\ref{A6}, we have $\pi_{\varepsilon}^{-1} \to \pi^{-1}_{\ast}$ in $L^{\infty}$. Therefore,
\begin{equation*}
    \begin{aligned}
        \lim_{\varepsilon \to 0}K_{\varepsilon}(\mu, \mu^{n - 1})^{2} = \lim_{\varepsilon \to 0} \int_{\mathbb{T}^{d}} \frac{(\mu - \mu^{n - 1})^{2}}{\pi_{\varepsilon}} \ud x = \int_{\mathbb{T}^{d}} \frac{(\mu - \mu^{n - 1})^{2}}{\pi_{\ast}} \ud x = K(\mu, \mu^{n - 1})^{2}.
    \end{aligned}
\end{equation*}
To conclude the proof, we show that
\begin{equation}\label{converge-K}
    \lim_{\varepsilon \to 0} K_{\varepsilon}(\mu_{\varepsilon}, \mu_{\varepsilon}^{n - 1}) - K_{\varepsilon}(\mu, \mu^{n - 1}) = 0.
\end{equation}
By using the triangle inequality of the weighted distance $K_{\varepsilon}(\mu, \nu)$, we have
\begin{equation*}
    \begin{aligned}
        |K_{\varepsilon}(\mu_{\varepsilon}, \mu_{\varepsilon}^{n - 1}) - K_{\varepsilon}(\mu, \mu^{n - 1})| &\le \big|K_{\varepsilon}(\mu_{\varepsilon}, \mu_{\varepsilon}^{n - 1}) - K_{\varepsilon}(\mu_{\varepsilon}, \mu^{n - 1})|\\
        &\quad + |K_{\varepsilon}(\mu_{\varepsilon}, \mu^{n - 1}) - K_{\varepsilon}(\mu, \mu^{n - 1})\big|\\
        &\le K_{\varepsilon}(\mu_{\varepsilon}^{n - 1}, \mu^{n - 1}) + K_{\varepsilon}(\mu_{\varepsilon}, \mu)\\
        &\le e^{\lambda/2} \|\mu_{\varepsilon}^{n - 1} - \mu^{n - 1}\|_{L^{2}} + e^{\lambda/2} \|\mu_{\varepsilon} - \mu\|_{L^{2}},
    \end{aligned}
\end{equation*}
where the last inequality is due to the fact that $\pi_{\varepsilon}^{-1}$ and $\pi_{\ast}^{-1}$ are bounded above by $e^{\lambda}$ according to Assumption~\ref{A4}. Hence, we have~\eqref{converge-K} since $\mu_{\varepsilon}^{n - 1} \to \mu^{n - 1}$ and $\mu_{\varepsilon} \to \mu$ in $L^{2}(\mathbb{T}^{d})$.
\end{proof}

\subsection{Compactness}\label{sec:2-2}
To conduct the homogenization of the solution $\mu_{\varepsilon, \tau}(t)$ to the weighted $L^{2}$ scheme, we establish the following $H^{1}$ estimate which gives us the compactness of $(\mu_{\varepsilon, \tau})_{\varepsilon > 0}$ in $L^{2}$.
\begin{lemma}[$H^{1}$ estimates]\label{u-compact}Given $\varepsilon, \tau > 0$. Let $\mu_{\varepsilon, \tau}(t)$ be the time-discrete solution of the weighted $L^{2}$ scheme in Definition~\ref{L2-scheme-def} with initial value $\mu_{0} \in H^{1}(\mathbb{T}^{d})$, then there exists a constant $C$ independent of $\varepsilon$ and $\tau$ such that
\begin{equation*}
    \|\mu_{\varepsilon, \tau}(t)\|_{H^{1}} \le C \text{ for all } t \in (0, \infty).
\end{equation*}
\end{lemma}
By combining Theorem~\ref{Gamma-i-lemma} and Lemma~\ref{u-compact}, we prove the homogenization theorem in Theorem~\ref{homo-of-ie}. The proof of Lemma~\ref{u-compact} will be put after the proof of Theorem~\ref{homo-of-ie}.

\begin{proof}[Proof of Theorem~\ref{homo-of-ie}]
Fix a time node $t^n$. By Lemma~\ref{u-compact}, we have
\begin{equation*}
    \|\mu_{\varepsilon}^{n}\|_{H^{1}} = \|\mu_{\varepsilon, \tau}(t_{n})\|_{H^{1}} \le C \,, \quad \forall \varepsilon > 0.
\end{equation*}
Since the embedding $H^{1}(\mathbb{T}^{d}) \hookrightarrow L^{2}(\mathbb{T}^{d})$ is compact, the family
$\{\mu_{\varepsilon}^{n}\}_{\varepsilon>0}$ is precompact in $L^{2}(\mathbb{T}^{d})$. Hence, there exists $\mu^{n}_{\ast} \in L^{2}(\mathbb{T}^{d})$ such that, along a subsequence $\varepsilon_\ell \to 0$,
\begin{equation}\label{compact-u-1}
    \mu_{\varepsilon_{\ell}}^{n} \to \mu^{n}_{\ast} \text{ in } L^{2}(\mathbb{T}^{d}) \text{ as } \ell \to \infty\,.
\end{equation}
It remains to identify the limit $\mu^{n}_{\ast}$. Since $\mu_{\varepsilon}^{0}=\mu^{0}$, the $\Gamma$-convergence result in Theorem~\ref{Gamma-i-lemma}, together with the fundamental theorem of $\Gamma$-convergence, Theorem~\ref{fund-G}, implies that $\mu^{1}_{\ast}=\mu^{1}$, where 
\begin{equation*}
    \mu^{1} = \arg\min_{\mu}\mathcal{I}_{\ast, \tau}(\mu, \mu^{0}) \text{ over } \mu \in H^{1}(\mathbb{T}^{d}).
\end{equation*}
In addition, according to the uniqueness of $\mu^{1}$, we have $\mu_{\varepsilon}^{1} \to \mu^{1}$ as $\varepsilon \to 0$, independent of the choice of a subsequence $(\varepsilon_{\ell})$. Apply this argument iteratively at each time step, we have $\mu^n_\ast = \mu^n$ with $\mu^n$ is defined in \eqref{limit-BE}, and
\begin{equation*}
    \mu_{\varepsilon, \tau}(t) \to \mu_{\tau}(t) \text{ in } L^{2}(\mathbb{T}^{d}) \text{ as $\varepsilon \to 0$ for all $t$,}
\end{equation*}
according to the definition of the time discrete trajectory $\mu_{\tau}(t)$ and $\mu_{\varepsilon, \tau}(t)$ from $\mu^n$ and $\mu_\varepsilon^n$, respectively in \eqref{L2-mu} and \eqref{0708}.

\end{proof}

To establish the $H^{1}$ estimate in Lemma~\ref{u-compact}. Let us prove the following.
\begin{lemma}\label{grad-esti-u}Given $\varepsilon, \tau > 0$, and let $\mu_{\varepsilon, \tau}$ be the time-discrete solution of the weighted $L^{2}$ scheme in Definition~\ref{L2-scheme-def}. Let $\varrho_{\varepsilon, \tau}(t) = \mu_{\varepsilon, \tau}(t)/\pi_{\varepsilon}$, there exists a constant $C$ independent of $\varepsilon$ and $\tau$ such that
\begin{equation*}
    \|\nabla \varrho_{\varepsilon, \tau}(t)\|_{L^{2}} \le C \text{ for all } t \in (0, \infty).
\end{equation*}
\end{lemma}
\begin{proof}
It is sufficient to prove this estimate at each time node $t = t_{n}$ for time-discrete solutions. Let $\mu_{\varepsilon}^{n} = \mu_{\varepsilon, \tau}(t_{n})$ for all $n$. Since $\mu_{\varepsilon}^{n}$ minimizes $\mathcal{I}_{\varepsilon, \tau}(\mu, \mu_{\varepsilon}^{n - 1})$, we have
\begin{equation*}
    Q_{\varepsilon}(\mu_{\varepsilon}^{n}) \le Q_{\varepsilon}(\mu_{\varepsilon}^{n - 1}) + \frac{1}{2\tau} K_{\varepsilon}(\mu_{\varepsilon}^{n - 1}, \mu_{\varepsilon}^{n - 1})^{2} = Q_{\varepsilon}(\mu_{\varepsilon}^{n - 1}).
\end{equation*}
By iterating the above inequality to the initial value, it implies that
\begin{equation}\label{Qbound}
    Q_{\varepsilon}(\mu_{\varepsilon}^{n}) \le Q_{\varepsilon}(\mu_{0}) \text{ for all } n.
\end{equation}
Denote $\varrho_{\varepsilon}^{n} = \mu_{\varepsilon}^{n}/\pi_{\varepsilon}$. Since $\Lambda_{0} I_{d \times d} \preceq A(x)$ and $\pi_{\varepsilon} \le 1$ by Assumptions~\ref{A3} and~\ref{A5}, we have
\begin{equation}\label{grad-bound-r}
    \begin{aligned}
        \int_{\mathbb{T}^{d}} |\nabla \varrho_{\varepsilon}^{n}|^{2} \ud x \le \Lambda_{0} \int_{\mathbb{T}^{d}} \Big<\pi_{\varepsilon}(x)A\Big(\frac{x}{\varepsilon}\Big) \nabla \varrho_{\varepsilon}^{n}, \nabla \varrho_{\varepsilon}^{n}\Big> \ud x = 2\Lambda_{0} Q_{\varepsilon}(\mu_{\varepsilon}^{n}) \le 2\Lambda_{0} Q_{\varepsilon}(\mu_{0}).
    \end{aligned}
\end{equation}
We conclude the estimate by showing that $Q_{\varepsilon}(\mu_{0})$ can be bounded from above by a constant independent of $\varepsilon$. By using a direct calculation, we have
\begin{equation*}
    Q_{\varepsilon}(\mu_{0}) = \frac{1}{2}\int_{\mathbb{T}^{d}} \bigg<\frac{1}{\pi_{\varepsilon}(x)}A\Big(\frac{x}{\varepsilon}\Big)\big(\nabla \mu_{0} + \mu_{0} \nabla V_{\varepsilon}\big), \nabla \mu_{0} + \mu_{0} \nabla V_{\varepsilon}\bigg> \ud x.
\end{equation*}
As $A(x) \preceq \Lambda_{1} I_{d \times d}$ and $\pi_{\varepsilon}^{-1} \le e^{\lambda}$ for $\pi_{\varepsilon}^{-1} = e^{V_{\varepsilon}}$ by Assumptions~\ref{A3} and~\ref{A4}, there exists a constant $C$ such that
\begin{equation*}
    \begin{aligned}
        Q_{\varepsilon}(\mu_{0}) &\le \frac{e^{\lambda}\Lambda_{1}}{2}\int_{\mathbb{T}^{d}} |\nabla \mu_{0} + \mu_{0} \nabla V_{\varepsilon}|^{2} \ud x \le e^{\lambda}\Lambda_{1}\int_{\mathbb{T}^{d}} |\nabla \mu_{0}|^{2} + |\mu_{0} \nabla V_{\varepsilon}|^{2} \ud x\\
        &\le C\|\mu_{0}\|_{H^{1}}^{2},
    \end{aligned}
\end{equation*}
which concludes the prove since $\mu_{0} \in H^{1}(\mathbb{T}^{d})$.
\end{proof}

Finally, we pass the estimate of $\|\nabla \varrho_{\varepsilon, \tau}\|_{L^{2}}$ (Lemma~\ref{grad-esti-u}) to the estimate of $\|\mu_{\varepsilon}\|_{H^{1}}$ in Lemma~\ref{u-compact}. We recall the following weighted Poincar\'e inequality (see, for example, \cite{pechstein2013weighted}).
\begin{lemma}[weighted Poincar\'e inequality]\label{wpi}Let $\Omega$ be open bounded domain with Lipchitz boundary, and let $\nu \in L^{\infty}(\Omega)$ be a weight function with $\int_{\Omega} \nu \ud x = 1$ and $\inf_{x \in \Omega} \nu(x) > 0$, there exists a positive constant $C$ depending on $\nu$ and $\Omega$ such that
\begin{equation*}
    \|u - \overline{u}\|_{L^{2}(\nu)}^{2} \le C\|\nabla u\|_{L^{2}(\nu)}^{2} \text{ where } \overline{u} = \int_{\Omega} u(x) \nu(x) \ud x.
\end{equation*}
Here the weighted $L^{2}$ norm is defined by $\|u\|_{L^{2}(\nu)} = \big[\int_{\Omega} |u(x)|^{2} \nu(x) \ud x\big]^{\frac{1}{2}}$.
\end{lemma}

\begin{proof}[Proof of Lemma~\ref{u-compact}]Let $\varrho_{\varepsilon, \tau}(t) = \mu_{\varepsilon, \tau}(t)/\pi_{\varepsilon}$ where $\pi_{\varepsilon}(x) = e^{-V_{\varepsilon}(x)}$. Consider the following weight,
\begin{equation*}
    \nu_\varepsilon := \frac{1}{Z_{\varepsilon}}\pi_{\varepsilon}(x) \text{ where } Z_{\varepsilon} = \int_{\mathbb{T}^{d}} \pi_{\varepsilon}(x) \ud x.
\end{equation*}
By using the weighted Poincar\'e inequality in Lemma~\ref{wpi} with the weight $\nu_\varepsilon$, there exists a positive constant $C$ such that 
\begin{equation}\label{wp_ineq}
    \begin{aligned}
        \int_{\mathbb{T}^{d}} |\nabla \varrho_{\varepsilon, \tau}(x, t)|^{2} \nu_\varepsilon(x) \ud x &\ge C \int_{\mathbb{T}^{d}} |\varrho_{\varepsilon, \tau}(x, t) - \overline{\varrho}_{\varepsilon, \tau}(t)|^{2} \nu_\varepsilon(x) \ud x\\
        &= C \int_{\mathbb{T}^{d}} |\varrho_{\varepsilon, \tau}(x, t)|^{2} \nu_\varepsilon(x) \ud x - C |\overline{\varrho}_{\varepsilon, \tau}(t)|^{2}.
    \end{aligned}
\end{equation}
Considering the mass conservation of time-discrete solutions in~\eqref{structure-preserving}, we have
\begin{equation*}
    \overline{\varrho}_{\varepsilon, \tau}(t) := \frac{1}{Z_{\varepsilon}}\int_{\mathbb{T}^{d}} \varrho_{\varepsilon, \tau}(x, t) \pi_{\varepsilon}(x) \ud x = \frac{1}{Z_{\varepsilon}}\int_{\mathbb{T}^{d}} \mu_{\varepsilon, t}(x, t) \ud x = \frac{1}{Z_{\varepsilon}}.
\end{equation*}
Since $0 \le V_{\varepsilon}(x) \le \lambda$, we have $Z_{\varepsilon} \ge e^{-\lambda}$. Hence, there exists two constant $C_{0}$ and $C_{1}$ (independent of $\varepsilon$) such that
\begin{equation*}
    \|\varrho_{\varepsilon, \tau}(t)\|_{L^{2}}^{2} \le C_{0} + C_{1} \|\nabla \varrho_{\varepsilon, \tau}(t)\|^{2}_{L^{2}}.
\end{equation*}
By using Lemma~\ref{grad-esti-u}, $\|\varrho_{\varepsilon, \tau}(t)\|_{H^{1}}$ is uniformly bounded from above.

Since $\{V_{\varepsilon}\}$ is uniformly bounded in $C^{1}(\mathbb{T}^{d})$ due to Assumption~\ref{A5}, there exists a constant $C$ such that
\begin{equation*}
    \|\nabla\mu_{\varepsilon, \tau}(t)\|_{L^{2}}^{2} = \int_{\mathbb{T}^{d}} |\nabla \mu_{\varepsilon, \tau}(t)|^{2} \ud x = \int_{\mathbb{T}^{d}} |\nabla (\varrho_{\varepsilon, \tau}(t)\pi_{\varepsilon})|^{2} \ud x \le C\|\varrho_{\varepsilon, \tau}(t)\|_{H^{1}}.
\end{equation*}
Finally, by using classical Poincar\'e inequality and the mass conservation~\eqref{structure-preserving}, there exists a positive constant $C$ such that
\begin{equation*}
    \|\mu_{\varepsilon, \tau}(t)\|_{L^{2}}^{2} - 1 = \int_{\mathbb{T}^{d}} |\mu_{\varepsilon, \tau}(t) - \overline{\mu}_{\varepsilon, \tau}(t)|^{2} \ud x \le C\int_{\mathbb{T}^{d}} |\nabla \mu_{\varepsilon, \tau}(t)|^{2} \ud x.
\end{equation*}
We can conclude the proof because both $\|\mu_{\varepsilon, \tau}(t)\|_{L^{2}}$ and $\|\nabla\mu_{\varepsilon, \tau}(t)\|_{L^{2}}$ are uniformly bounded.
\end{proof}

\section{$\varepsilon$-JKO scheme}\label{sec3}
In this section, we analyze the $\varepsilon$-JKO scheme constructed in Definition~\ref{JKO-scheme-def}, where the solutions are defined according to
\begin{equation}\label{JKO-WW-copy}
    \min_{P(\mathbb{T}^{d})} \mathcal{J}_{\varepsilon, \tau}(\mu, \mu^{n - 1}_{\varepsilon}) := H_{\varepsilon}(\mu) + \frac{1}{2\tau} W_{\varepsilon}(\mu, \mu^{n - 1}_{\varepsilon})^{2}.
\end{equation}
The existence of a unique minimizer was proved in \cite{lisini2009nonlinear, liu2025variational}.
\begin{proposition}[well-posedness]\label{JKO-existence}Given $\varepsilon, \tau > 0$ and $\mu_{0} \in P(\mathbb{T}^{d})$, there exists a unique solution $\mu_{\varepsilon, \tau}$ to the $\varepsilon$-JKO scheme in Definition~\ref{JKO-scheme-def} with initial value $\mu_{0}$.
\end{proposition}

Unlike in the previous section, and in particular unlike Proposition~\ref{IE-time}, the convergence in the limit $\tau \to 0$ is less standard in the literature and requires a more involved analysis. We therefore separate the discussion into two subsections. The first subsection focuses on the homogenization limit, namely $\varepsilon \to 0$ with $\tau$ fixed, while the second subsection addresses the convergence as $\tau \to 0$ for the homogenized limiting scheme.

\subsection{Homogenization}
In this subsection, we prove the homogenization of the $\varepsilon$-JKO scheme in Theorem~\ref{homo-of-jko}. Let us establish the following $\Gamma$-convergence.
\begin{theorem}\label{JKO-e-theorem}
If $\mu^{n - 1}_{\varepsilon} \to \overline{\mu}^{n - 1}$  in $W_{2}$, we have the following $\Gamma$-convergence (as functional over $\mu$) in the $W_{2}$ metric,
\begin{equation}\label{JKO-e}
    \Gamma\text{-}\lim_{\varepsilon \to 0} \mathcal{J}_{\varepsilon, \tau}(\mu, \mu^{n - 1}_{\varepsilon}) = \mathcal{J}_{\ast, \tau}(\mu, \overline{\mu}^{n - 1}) := H(\mu) + \frac{1}{2\tau} W_{\mathrm{GH}}(\mu, \overline{\mu}^{n - 1})^{2}
\end{equation}
where $H(\mu)$ and $W_{\mathrm{GH}}(\mu, \nu)$ are defined in \eqref{limiting_W}.
\end{theorem}

We show the well-posedness of the limiting JKO scheme based on \cite{jordan1998variational, lisini2009nonlinear, liu2025variational}.
\begin{proposition}[well-posedness]\label{HJKO-existence}Given $\tau > 0$ and $\overline{\mu}^{n - 1} \in P(\mathbb{T}^{d})$, then there is a unique $\overline{\mu}^{n} \in P(\mathbb{T}^{d})$ that minimizes the $\mathcal{J}_{\ast, \tau}(\mu, \overline{\mu}^{n - 1})$ . Hence, the homogenized JKO scheme in Theorem~\ref{homo-of-jko} is well-posed.
\end{proposition}
\begin{proof}
By~\eqref{B-ast-1} and Lemmas~\ref{c-eqi} and~\ref{w-eqi}, $W_{\mathrm{GH}}(\mu, \nu)$ is equivalent to $W_{2}(\mu, \nu)$, which means that $W_{\mathrm{GH}}(\mu, \nu)$ is continuous, i.e.,
\begin{equation*}
    W_{\mathrm{GH}}(\mu, \overline{\mu}^{n - 1}) = \lim_{\ell \to \infty} W_{\mathrm{GH}}(\mu_{\ell}, \overline{\mu}^{n - 1})
\end{equation*}
for all $\mu_{\ell} \to \mu$ in $W_{2}$ metric.

Let $(\mu_{\ell})_{\ell \in \mathbb{N}} \subseteq P(\mathbb{T}^{d})$ be a minimizing sequence for $\mathcal{J}_{\ast, \tau}(\mu, \overline{\mu}^{n - 1})$. Since $P(\mathbb{T}^{d})$ is compact in $W_{2}$ metric topology, there exists a limit $\overline{\mu}^{n}$ such that $\mu_{\ell} \to \overline{\mu}^{n}$ as $\ell \to \infty$, up to the extraction of a subsequence. Hence, by considering the lower-semi-continuity of $H(\mu)$ in $P(\mathbb{T}^{d})$, we have
\begin{equation*}
    \mathcal{J}_{\ast, \tau}(\overline{\mu}^{n}, \overline{\mu}^{n - 1}) \le \liminf_{\ell \to \infty} \mathcal{J}_{\ast, \tau}(\mu_{\ell}, \overline{\mu}^{n - 1}) = \inf_{\mu}\mathcal{J}_{\ast, \tau}(\mu, \overline{\mu}^{n - 1}),
\end{equation*}
which means that $\overline{\mu}^{n}$ is the desired minimal. The uniqueness of $\overline{\mu}^{n}$ is given by the strict convexity of $H(\mu)$.
\end{proof}

Similar to the mathematical induction in Section~\ref{sec:2-2}, we prove Theorem~\ref{homo-of-jko} based on the $\Gamma$-convergence in Theorem~\ref{JKO-e-theorem}.
\begin{proof}[Proof of Theorem~\ref{homo-of-jko}] 
Since the space $(P(\mathbb{T}^{d}), W_{2})$ is compact (Lemma~\ref{comp-1}), the family of solutions $\{\mu_{\varepsilon}^{n}\}_{\varepsilon > 0}$ is pre-compact, and there exists $\{\overline{\mu}^{n}_{\ast}\}_{n \in \mathbb{N}}$ such that, along a subsequence $\varepsilon_{\ell} \to 0$
\begin{equation*}
    \mu_{\varepsilon_{\ell}}^{n} \to \mu_{\ast}^{n} \text{ in $W_{2}$ as } \ell \to \infty \text{ for all } n \in \mathbb{N},
\end{equation*}
To identify the limit $\overline{\mu}^{n}_{\ast}$, we invoke the $\Gamma$-convergence result in Theorem~\ref{JKO-e-theorem}, together with the fundamental theorem of $\Gamma$-convergence, Theorem~\ref{fund-G}, to get $\mu_{\ast}^{1} = \overline{\mu}^{1}$, where
\begin{equation*}
    \overline{\mu}^{1} = \arg\min_{\mu} \mathcal{J}_{\ast, \tau}(\mu, \mu^{0}) \text{ over } \mu \in P(\mathbb{T}^{d}).
\end{equation*}
In addition, according to the uniqueness of $\overline{\mu}^{1}$, we have $\overline{\mu}_{\varepsilon}^{1} \to \overline{\mu}^{1}$ as $\varepsilon \to 0$, independent of the choice of a subsequence. Apply this argument iteratively at each time step, we have $\mu_{\ast}^{n} = \overline{\mu}^{n}$ with $\overline{\mu}^{n}$ defined in~\eqref{limiting-JKO-def}, and
\begin{equation*}
    \text{$\mu_{\varepsilon, \tau}(t) \to \overline{\mu}_{\tau}(t)$ in $W_{2}$ as $\varepsilon \to 0$ for all $t$}.
\end{equation*}
according to the definition of the time discrete trajecotry $\overline{\mu}_{\tau}(t)$ and $\mu_{\varepsilon, \tau}(t)$ from $\overline{\mu}^{n}$ and $\mu_{\varepsilon}^{n}$, respectively in~\eqref{homo-jko-solution-def} and~\eqref{JKO-solution}.
\end{proof}

To prove the $\Gamma$-convergence in Theorem~\ref{JKO-e-theorem}, it is sufficient to establish the following two lemmas, by virtue of the stability of $\Gamma$-convergence under
continuous perturbations stated in Lemma~\ref{stability-Gamma-thm}. The proofs of these lemmas are provided in the next two subsections.
\begin{lemma}\label{p2}
We have
\begin{equation*}
    \Gamma\textnormal{-}\lim_{\varepsilon \to 0} H_{\varepsilon} = H
\end{equation*}
in the $W_2$-metric, where $H_{\varepsilon}$ and $H$ are the entropy functions give in~\eqref{Hvare} and~\eqref{limiting_W}
\end{lemma}
\begin{lemma}\label{p1-1}The distances $W_{\varepsilon}$ and $W_{\mathrm{GH}}$ obey the following:
\begin{equation*}
    \lim_{\varepsilon \to 0}W_{\varepsilon}(\mu_{\varepsilon}, \mu_{\varepsilon}^{n - 1})^{2} = W_{\mathrm{GH}}(\mu, \mu^{n - 1})^{2}
\end{equation*}
for all sequences $\mu_{\varepsilon}$ and $\mu_{\varepsilon}^{n - 1}$ such that $\mu_{\varepsilon} \to \mu$ and $\mu_{\varepsilon}^{n - 1} \to \mu^{n - 1}$ in $W_{2}$.
\end{lemma}

We note that our result on the homogenization of $\varepsilon$-Wasserstein distances largely follows from \cite[Section 5]{gao2023homogenization}, and Gangbo and Tudorascu \cite{gangbo2012homogenization} in which the homogenization was established as the corollary of homogenization of Hamilton-Jacobi equations in Wasserstein spaces in terms of the Benamou-Brenier dynamical formulation of $W_{\varepsilon}$. However, their discussion are mainly focused on the cases where the spatial domain is $\mathbb{R}^{d}$, and the end-points remain fixed. In the following, we provide a rigorous derivation of the homogenization on $\mathbb{T}^{d}$, and consider the perturbation of end-points coming from the solution of JKO schemes.

\subsubsection{Convergence of entropy}
\begin{proof}[Proof of Lemma~\ref{p2}]
 Consider the entropy functionals of the following form; see, for example, \cite[Chapter 9.4]{ambrosio2008gradient}:
\begin{equation*}
    \text{$H_{\varepsilon}(\mu) = H(\mu|\pi_{\varepsilon}\ud x)$ and $H(\mu) = H(\mu|\pi_{\ast} \ud x)$,}
\end{equation*}
where
\begin{equation*}
    H(\mu | \nu) =
    \begin{cases}
        \int_{\mathbb{T}^{d}} \varrho \log(\varrho)\ud \nu & \text{ if } \text{d}\mu = \varrho(x)\text{d}\nu\\
        \infty & \text{ otherwise. }
    \end{cases}
\end{equation*}
By Assumption~\ref{A6}, we have $\pi_{\varepsilon} \to \pi_{\ast}$ uniformly. Hence,
\begin{equation*}
    \lim_{\varepsilon \to 0}\int_{\mathbb{T}^{d}} \varphi(x) \pi_{\varepsilon}(x) \ud x = \int_{\mathbb{T}^{d}} \varphi(x) \pi_{\ast}(x) \ud x \text{ for all } \varphi \in C(\mathbb{T}^{d})
\end{equation*}
Therefore, $\pi_{\varepsilon}\ud x$ converges narrowly to $\pi_{\ast}\ud x$.

By the joint lower semicontinuity of the relative entropy on compact spaces \cite[Lemma 9.4.3]{ambrosio2008gradient}, for any sequence $\{\mu_{\varepsilon}\}$ converging to $\mu$ in $W_{2}$, we obtain
\begin{equation*}
    \liminf_{\varepsilon \to 0} H_{\varepsilon}(\mu_{\varepsilon}) \ge H(\mu).
\end{equation*}

Now, for given $\mu$, we let $\mu_{\varepsilon} = \mu$ for all $\varepsilon > 0$ be the recovery sequence. By using the fact that
\begin{equation*}
    \lim_{\varepsilon \to 0} \bigg|\int_{\mathbb{T}^{d}} \big(V_{\varepsilon}(x) - V(x)\big) \ud \mu\bigg| \le \lim_{\varepsilon \to 0} \sup_{x \in \mathbb{T}^{d}} \big|V_{\varepsilon}(x) - V(x)\big| \mu(\mathbb{T}^{d}) = 0,
\end{equation*}
we have
\begin{equation*}
    \begin{aligned}
        \limsup_{\varepsilon \to 0} H_{\varepsilon}(\mu_{\varepsilon}) &= S(\mu) + \limsup_{\varepsilon \to 0} \int_{\mathbb{T}^{d}} V_{\varepsilon}(x) \ud \mu = H(\mu),
    \end{aligned}
\end{equation*}
which gives $\Gamma\textnormal{-}\lim_{\varepsilon \to 0} H_{\varepsilon} = H$.
\end{proof}

\subsubsection{Homogenization of distance}
To prove the homogenization of distance, we first establish the following point-wise convergence of cost functions. 
\begin{lemma}\label{converge-c-lemma}
Given $(x, y) \in \mathbb{T}^{d} \times \mathbb{T}^{d}$, we have $\lim_{\varepsilon \to 0} c_{\varepsilon}(x,y) = c_{\mathrm{hom}}(x,y)$.
\end{lemma}

\begin{proof}
Let $E_{\varepsilon}(z) := \int_0^1 \big<A(z/\varepsilon)^{-1}\dot{z}, \dot{z}\big> \ud t$ and $E_{\ast}(z) := \int_{0}^{1}A_{\mathrm{hom}}(\dot{z})\ud t$ where $A_{\mathrm{hom}}$ is defined in~\eqref{Bast}. According to the definition of least action, we write
\begin{equation}\label{eps-c-copy}
    c_\eps(x,y) = \min_{z \in \mathscr{C}(x, y)} E_{\varepsilon}(z) \text{ and } c_{\mathrm{hom}}(x,y) = \min_{z \in \mathscr{C}(x, y)} E_{\ast}(z)
\end{equation}
where $\mathscr{C}(x, y)$ is the space of curves connecting $x$ and $y$ on $\mathbb{T}^{d}$ defined by
\begin{equation*}
    \mathscr{C}(x, y) := \Big\{z \in H^{1}\big([0, 1]; \mathbb{T}^{d}\big): z(0) = x \textnormal{ and } z(1)=y\Big\}.
\end{equation*}
By periodically extending $\mathbb{T}^{d} = [0, 1]^{d}$ to $\mathbb{R}^{d}$ through the universal Riemannian cover, we can compute $c_\eps(x,y)$ via
\begin{equation}\label{eps-c-copy-extend}
    c_\eps(x,y) = \min_{z \in \mathscr{C}_{\text{cov}}(x, y)}E_{\varepsilon}(z)
\end{equation}
where the minimization happens in the following admissible set,
\begin{equation*}
    \mathscr{C}_{\text{cov}}(x, y) := \Big\{z \in H^{1}([0, 1]; \mathbb{R}^{d}): z(0) = x \text{ and } z(1)\in y + \mathbb{Z}^{d}\Big\}.
\end{equation*}
Since Riemannian cover preserves the local metric, we have
\begin{equation*}
    c_\varepsilon(x,y) = \min_{z \in \mathscr{C}(x, y)}E_{\varepsilon}(z) = \min_{z \in \mathscr{C}_{\mathrm{cov}}(x, y)} E_{\varepsilon}(z).
\end{equation*}
We now apply homogenization of vector-valued functionals \cite{weinan1991class, braides2002riemannian, braides2002gamma} on $\mathscr{C}_{\mathrm{cov}}(x, y)$ to get the following $\Gamma$-convergence,
\begin{equation}\label{G-be}
    \Gamma\text{-}\lim_{\varepsilon \to 0} E_{\varepsilon}(z) = E_{\ast}(z) = \int_{0}^{1}A_{\mathrm{hom}}(\dot{z})\ud t \text{ in $L^{2}([0, 1]; \mathbb{R}^{d})$, }
\end{equation}
for all $z \in \mathscr{C}_{\mathrm{cov}}(x, y)$.

To conclude the proof of Lemma~\ref{converge-c-lemma} based on the $\Gamma$-convergence~\eqref{G-be}, it suffices to establish the compactness of minimizing geodesics in $L^{2}([0, 1]; \mathbb{R}^{d})$. We let $z_{\varepsilon}$ be a minimizer of the action functional~\eqref{eps-c-copy-extend} for given $\varepsilon$. By Assumption~\ref{A3}, and equivalence of distances in Lemma~\ref{c-eqi}, we have
\begin{equation}\label{grad-w}
    \Lambda_{1}^{-1} \|\dot{z}_{\varepsilon}\|_{L^{2}([0, 1]; \mathbb{R}^{d})}^{2} \le E_{\varepsilon}(z_{\varepsilon}) := \int_0^1 \big<A(z_{\varepsilon}/\varepsilon)^{-1}\dot{z}_{\varepsilon}, \dot{z}_{\varepsilon}\big> \ud t \le \Lambda_{0}^{-1} d_{\mathrm{Eu}}(x, y)^{2},
\end{equation}
which means that $\|\dot{z}_{\varepsilon}\|_{L^{2}}^{2}$ is uniformly bounded. By H\"older's inequality \cite[Theorem 4.4]{struwe2000variational}, we have
\begin{equation}\label{Holder-w}
    \begin{aligned}
        |z_{\varepsilon}(t_{0}) - z_{\varepsilon}(t_{1})| &\le \int_{t_{0}}^{t_{1}} |\dot{z}_{\varepsilon}(t)| \ud t \le \bigg(|t_{1} - t_{0}| \int_{t_{0}}^{t_{1}} |\dot{z}_{\varepsilon}(t)|^{2} \ud t\bigg)^{\frac{1}{2}}\\
        &\le |t_{1} - t_{0}|^{\frac{1}{2}} \Lambda_{0}^{-\frac{1}{2}} d_{\mathrm{Eu}}(x, y),
    \end{aligned}
\end{equation}
which means that $(z_{\varepsilon})_{\varepsilon > 0}$ is equi-continuous. Since $z_{\varepsilon}(0) = x$, we also have
\begin{equation*}
    |z_{\varepsilon}(t) - x| \le t^{\frac{1}{2}} \Lambda_{0}^{-\frac{1}{2}} d_{\mathrm{Eu}}(x, y) \text{ for } 0 \le t \le 1 \text{ and } \varepsilon > 0,
\end{equation*}
which means that $(z_{\varepsilon})_{\varepsilon > 0}$ is uniformly bounded.

In summary, $\|z_{\varepsilon}\|_{H^{1}}$ is uniformly bounded from above, which means that $(z_{\varepsilon})_{\varepsilon > 0}$ is precompact in $L^{2}([0, 1]; \mathbb{R}^{d})$. Hence, by further using Arz\'ela-Ascoli's theorem, there exists $z_{\ast} \in H^{1}([0, 1]; \mathbb{R}^{d})$ such that
\begin{equation*}
    \text{$z_{\varepsilon} \to z_{\ast}$ both uniformly and in $L^{2}([0, 1], \mathbb{R}^{d})$ as $\varepsilon \to 0$,} 
\end{equation*}
up to the extraction of a subsequence. The uniform convergence implies $z_{\ast}(1) \in y + \mathbb{Z}^{d}$ which means that $z_{\ast}$ is admissible for~\eqref{eps-c-copy-extend}.

Finally, by fundamental theorem (Theorem~\ref{fund-G}), we have
\begin{equation*}
    \min_{z \in \mathscr{C}_{\mathrm{cov}}(x, y)}E_{\varepsilon}(z) \to \min_{z \in \mathscr{C}_{\mathrm{cov}}(x, y)}E_{\ast}(z) \text{ as } \varepsilon \to 0.
\end{equation*}
Therefore,
\begin{equation*}
    \begin{aligned}
        c_\eps(x,y) = \min_{z \in \mathscr{C}(x, y)}E_{\varepsilon}(z) &= \min_{z \in \mathscr{C}_{\mathrm{cov}}(x, y)}E_{\varepsilon}(z)\\
        &\to \min_{z \in \mathscr{C}_{\mathrm{cov}}(x, y)} E_{\ast}(z) = \min_{z \in \mathscr{C}(x, y)}E_{\ast}(z) = c_\ast(x,y),
    \end{aligned}
\end{equation*}
which is the convergence of costs.
\end{proof}

We now establish the convergence of Wasserstein distances in Lemma~\ref{p1-1} based on the point-wise convergence of cost functions in Lemma~\ref{converge-c-lemma}. The proof is due to the stability of optimal transport presented in \cite[Theorem 28.6]{villani2008optimal} and \cite[Section 5.2]{gao2023homogenization}.
\begin{proof}[Proof of Lemma~\ref{p1-1}]We first show that, for given $\mu$ and $\mu^{n - 1}$,
\begin{equation*}
    \lim_{\varepsilon \to 0} W_{\varepsilon}(\mu, \mu^{n - 1})^{2} = W_{\mathrm{GH}}(\mu, \mu^{n - 1})^{2}.
\end{equation*}
This amounts to prove that, for any $\delta > 0$, there exists $\varepsilon(\delta) > 0$ such that
\begin{equation}\label{p1-1-2}
    \big|W_{\varepsilon}(\mu, \mu^{n - 1})^{2} - W_{\mathrm{GH}}(\mu, \mu^{n - 1})^{2}\big| < \delta \text{ for all } \varepsilon < \varepsilon(\delta),
\end{equation}
In fact, by the point-wise convergence $c_{\varepsilon}(x, y) \to c_{\mathrm{hom}}(x, y)$, for any $\delta > 0$, there exists $\varepsilon(\delta) > 0$ such that
\begin{equation*}
    |c_{\varepsilon}(x, y) - c_{\mathrm{hom}}(x, y)| < \delta \text{ for all } \varepsilon < \varepsilon(\delta).
\end{equation*}
Since $\mathbb{T}^{d} \times \mathbb{T}^{d}$ is compact, $\varepsilon(\delta)$ can be made independent of $(x, y)$.

Let $\omega_{\varepsilon}$ be the optimal coupling in the distance $W_{\varepsilon}(\mu, \mu^{n - 1})$, we have
\begin{equation*}
    \begin{aligned}
        W_{\mathrm{GH}}(\mu, \mu^{n - 1})^{2} \le \int_{\mathbb{T}^{d} \times \mathbb{T}^{d}} c_{\mathrm{hom}}(x, y) \ud \omega_{\varepsilon} &\le \int_{\mathbb{T}^{d} \times \mathbb{T}^{d}} c_{\varepsilon}(x, y) \ud \omega_{\varepsilon} + \delta\\
        &= W_{\varepsilon}(\mu, \mu^{n - 1})^{2} + \delta.
    \end{aligned}
\end{equation*}
By switching the role between $W_{\mathrm{GH}}$ and $W_{\varepsilon}$ in the above argument, we can also get
\begin{equation*}
    W_{\varepsilon}(\mu, \mu^{n - 1})^{2} \le W_{\mathrm{GH}}(\mu, \mu^{n - 1})^{2} + \delta.
\end{equation*}
Hence, we have~\eqref{p1-1-2}.

To conclude the lemma, we need to show that, if $\mu_{\varepsilon} \to \mu$ and $\mu_{\varepsilon}^{n - 1} \to \mu^{n - 1}$ in $W_{2}$, then
\begin{equation}\label{dist-w-convergence}
    \lim_{\varepsilon \to 0} W_{\varepsilon}(\mu_{\varepsilon}, \mu^{n - 1}_{\varepsilon}) - W_{\varepsilon}(\mu, \mu^{n - 1}) = 0.
\end{equation}
Indeed, by using the triangle inequality for Wasserstein distances (see \cite[page 94]{villani2008optimal}), we have
\begin{equation}\label{triangle-1}
    \begin{aligned}
        \big|W_{\varepsilon}(\mu_{\varepsilon}, \mu_{\varepsilon}^{n - 1}) - W_{\varepsilon}(\mu, \mu^{n - 1})\big| &\le \big|W_{\varepsilon}(\mu_{\varepsilon}, \mu_{\varepsilon}^{n - 1}) - W_{\varepsilon}(\mu_{\varepsilon}, \mu^{n - 1})\big|\\
        &\quad + \big|W_{\varepsilon}(\mu_{\varepsilon}, \mu^{n - 1}) - W_{\varepsilon}(\mu, \mu^{n - 1})\big|\\
        &\le W_{\varepsilon}(\mu_{\varepsilon}^{n - 1}, \mu^{n - 1}) + W_{\varepsilon}(\mu_{\varepsilon}, \mu)\\
        &\le \Lambda_{0}^{-\frac{1}{2}} W_{2}(\mu^{n - 1}_{\varepsilon}, \mu^{n - 1}) + \Lambda_{0}^{-\frac{1}{2}} W_{2}(\mu_{\varepsilon}, \mu),
    \end{aligned}
\end{equation}
where the last inequality is due to the fact $c_{\varepsilon}(x, y) \le \Lambda_{0}^{-1} d_{\mathrm{Eu}}(x, y)^{2}$ and the equivalence of distances in Lemmas~\ref{c-eqi} and~\ref{w-eqi}. Since $\mu_{\varepsilon} \to \mu$ and $\mu^{n - 1}_{\varepsilon} \to \mu^{n - 1}$ in $W_{2}$, we have~\eqref{dist-w-convergence}.
\end{proof}

\subsection{Convergence in time}\label{sec:3-2}
In this subsection, we analyze the limit $\tau \to 0$ for both the $\varepsilon$-JKO scheme and the homogenized JKO scheme. We do not distinguish between a probability measure $\mu$ and its density, since all measures under consideration are absolutely continuous due to the presence of the logarithmic entropy.

\subsubsection{$\varepsilon$-JKO scheme}
We show the convergence of $\varepsilon$-JKO scheme in the vanishing step-size limit $\tau \to 0$ as follows.
\begin{theorem}\label{JKO-t}Given $\varepsilon > 0$. Let $\mu_{\varepsilon, \tau}$ be the solution of the $\varepsilon$-JKO scheme with initial value $\mu_{0}$ with step-size $\tau$ given in Definition~\ref{JKO-scheme-def}, we have,
\begin{equation*}
    \mu_{\varepsilon, \tau}(t) \to \mu_{\varepsilon}(t) \text{ in $W_{2}$-metric, ~for all $t \in (0, \infty)$ as $\tau \to 0$,}
\end{equation*}
where $\mu_{\varepsilon}(t)$ is the solution to~\eqref{0} with initial value $\mu_{0}$.
\end{theorem}

We note that the convergence analysis in the limit $\tau \to 0$ has been studied extensively for standard Wasserstein gradient flows; see, for example, \cite{ambrosio2008gradient}. By contrast, results for inhomogeneous Fokker--Planck equations with spatially varying diffusion coefficients are relatively scarce. Such convergence was first established by Lisini \cite{lisini2009nonlinear} using the theory of gradient flows in metric spaces developed in \cite{ambrosio2008gradient}, and was more recently proved by Liu and Tzavaras \cite{liu2025variational} through a Nash--Kuiper isometric embedding of the underlying metric.

In this work, we present an alternative proof based on the geometry of the distance $W_{\varepsilon}(\mu,\nu)$, as described in \cite[Chapter 10]{villani2008optimal}. This approach can be adapted to prove the convergence of the homogenized JKO scheme below. We note that the estimates in our proof share similarities with those used for the convergence of JKO schemes on Riemannian manifolds with bounded Ricci curvature \cite{erbar2010heat, rankin2024jko}.

We first show the following first-order optimality condition of the $\varepsilon$-JKO scheme.
For the ease of notation, we omit the subscript $\varepsilon$ in $\mu$ throughout this section.
\begin{proposition}\label{op-ejko}Given $\mu^{n - 1} \in P(\mathbb{T}^{d})$, let $\mu^{n} = \arg\min_{\mu} \mathcal{J}_{\varepsilon, \tau}(\mu, \mu^{n - 1})$ with $\mathcal J$ defined in \eqref{JKO-WW-copy}, 
and $\boldsymbol{\mathrm{r}}_{\ast}$ be the optimal transport map from $\mu^{n}$ to $\mu^{n - 1}$ under the cost $c_{\varepsilon}(x, y)$, we have
\begin{equation*}
    \begin{aligned}
        -\int_{\mathbb{T}^{d}} \nabla \cdot \boldsymbol{\mathrm{u}}(x) \ud \mu^{n} = \int_{\mathbb{T}^{d}} \frac{1}{\tau} \big<\boldsymbol{\mathrm{u}}(x), A(x/\varepsilon)^{-1}\xi_{x \to \boldsymbol{\mathrm{r}}_{\ast}(x)} &- \tau \nabla V_{\varepsilon}(x)\big> \ud \mu^{n}\\
        &\text{ for all } \boldsymbol{\mathrm{u}} \in C^{1}(\mathbb{T}^{d}; \mathbb{R}^{d}),
    \end{aligned}
\end{equation*}
where $\xi_{x \to y} = \dot{z}(0)$ is the initial velocity vector of an action minimizing curve $z(t)$ from $x$ to $y$ under the action $\int_{0}^{1}\la A(z/\varepsilon) \dot{z}, \dot{z}\ra \ud t$.
\end{proposition}

\begin{proof}[Proof of Proposition~\ref{op-ejko}]Let $\boldsymbol{\mathrm{u}} : \mathbb{T}^{d} \to \mathbb{R}^{d}$ be any $C^{1}$ test vector field. Consider the flow map $\boldsymbol{\mathrm{r}}(x, s)$ generated by $\boldsymbol{\mathrm{u}}$ defined as follows,
\begin{equation}\label{flow-def}
    \frac{\partial \boldsymbol{\mathrm{r}}(x, s)}{\partial s} = \boldsymbol{\mathrm{u}} \circ \boldsymbol{\mathrm{r}}(x, s),\quad \boldsymbol{\mathrm{r}}(\cdot, 0) = \text{id}. 
\end{equation}
Let $\mu(s) := \boldsymbol{\mathrm{r}}(\cdot, s) \# \mu^{n}$. By the optimality of $\mu^{n}$, we have
\begin{equation}\label{J-stability}
    \begin{aligned}
        &\frac{\mathcal{J}_{\varepsilon, \tau}(\mu(s), \mu^{n - 1}) - \mathcal{J}_{\varepsilon, \tau}(\mu^{n}, \mu^{n - 1})}{s}\\
        &= \frac{H_{\varepsilon}(\mu(s), \mu^{n - 1}) - H_{\varepsilon}(\mu^{n}, \mu^{n - 1})}{s} + \frac{W_{\varepsilon}(\mu(s), \mu^{n - 1}) - W_{\varepsilon}(\mu^{n}, \mu^{n - 1})}{2\tau \times s} \ge 0.
    \end{aligned}
\end{equation}
For the entropy functional, the limit $s \to 0$ is given by \cite[equation (37) and (38)]{jordan1998variational} as follows,
\begin{equation}\label{limH}
    \lim_{s \to 0}\frac{H_{\varepsilon}(\mu(s)) - H_{\varepsilon}(\mu^{n})}{s} = \int_{\mathbb{T}^{d}} \la\nabla V_{\varepsilon}, \boldsymbol{\mathrm{u}} \ra - \nabla \cdot \boldsymbol{\mathrm{u}} \ud \mu^{n}.
\end{equation}

To take the limit $s \to 0$ for the transport metric, we let $\omega$ be the optimal coupling between $\mu^{n}$ and $\mu^{n - 1}$ in the distance $W_{\varepsilon}(\mu^{n}, \mu^{n - 1})$. The flow map $\boldsymbol{\mathrm{r}}$ then induces a time dependent coupling $\omega[s]$ by pushing-forward the first marginal measure:
\begin{equation*}
    \int_{\mathbb{T}^{d} \times \mathbb{T}^{d}} \varphi(x, y) \ud \omega[s](x, y) = \int_{\mathbb{T}^{d} \times \mathbb{T}^{d}} \varphi(\boldsymbol{\mathrm{r}}(x, s), y)\ud \omega(x, y), \text{ for all } \varphi \in C(\mathbb{T}^{d}).
\end{equation*}
By using the fact that $W_{\varepsilon}(\mu(s), \mu^{n - 1})^{2} \le \int_{\mathbb{T}^{d} \times \mathbb{T}^{d}} c_{\varepsilon}(x, y) \ud \omega[s]$, we have
\begin{equation}\label{ineqW}
    \begin{aligned}
        \frac{1}{2s} \Big(W_{\varepsilon}(\mu(s), \mu^{n - 1})^{2} &- W_{\varepsilon}(\mu^{n}, \mu^{n - 1})^{2}\Big)\\
        &\le \int_{\mathbb{T}^{d} \times \mathbb{T}^{d}} \frac{1}{2s}\big(c_{\varepsilon}(\boldsymbol{\mathrm{r}}(x, s), y) - c_{\varepsilon}(x, y)\big) \ud \omega(x, y).
    \end{aligned} 
\end{equation}
According to \cite[Proposition 10.15]{villani2008optimal}, the cost $c_{\varepsilon}(x, y)$ given by the Lagrangian action function $\int_{0}^{1}(A(z/\varepsilon)^{-1}\dot{z}, \dot{z})\ud t$ is superdifferentiable:
\begin{equation*}
    \begin{aligned}
        \frac{c_{\varepsilon}(\boldsymbol{\mathrm{r}}(x, s), y) - c_{\varepsilon}(x, y)}{2s} &\le \frac{\big<\boldsymbol{\mathrm{r}}(x, s) - x, p\big>}{2s} + \frac{o(|\boldsymbol{\mathrm{r}}(x, s) - x|)}{2s},\\
        &\text{ with super-gradient } p = -2A(x/\varepsilon)^{-1} \xi_{x \to y},
    \end{aligned}
\end{equation*}
where $\xi_{x \to y}$ is the initial velocity of an action minimizing curve connecting $x$ to $y$ and $o(|\boldsymbol{\mathrm{r}}(x, s) - x|)/s \to 0$ as $s \to 0$. Therefore, considering~\eqref{ineqW}, we have
\begin{equation}\label{limW}
    \begin{aligned}
        \limsup_{s \to 0}\frac{1}{2s} \Big(W_{\varepsilon}(\mu(s), \mu^{n - 1})^{2} &- W_{\varepsilon}(\mu^{n}, \mu^{n - 1})^{2}\Big) \\
        &\le \int_{\mathbb{T}^{d} \times \mathbb{T}^{d}} -\big<\boldsymbol{\mathrm{u}}(x), A(x/\varepsilon)^{-1}\xi_{x \to y}\big> \ud \omega(x, y).
    \end{aligned} 
\end{equation}

Combining~\eqref{limH} and~\eqref{limW}, we can take $\limsup_{s \to 0}$ in~\eqref{J-stability} to get
\begin{equation*}
    \int_{\mathbb{T}^{d}} \la\nabla V_{\varepsilon}, \boldsymbol{\mathrm{u}} \ra - \nabla \cdot \boldsymbol{\mathrm{u}} \ud \mu^{n} - \frac{1}{\tau} \int_{\mathbb{T}^{d} \times \mathbb{T}^{d}} \big<\boldsymbol{\mathrm{u}}(x), A(x/\varepsilon)^{-1}\xi_{x \to y}\big> \ud \omega(x, y) \ge 0.
\end{equation*}
Let $\boldsymbol{\mathrm{r}}_{\ast}$ be the optimal transport map from $\mu^{n}$ to $\mu^{n - 1}$, since $\omega = (\mathrm{id}, \boldsymbol{\mathrm{r}}_{\ast}) \# \mu^{n}$, we can further get
\begin{equation*}
    \begin{aligned}
        \int_{\mathbb{T}^{d}} - \nabla \cdot \boldsymbol{\mathrm{u}} - \frac{1}{\tau}\big<\boldsymbol{\mathrm{u}}(x), A(x/\varepsilon)^{-1}\xi_{x \to \boldsymbol{\mathrm{r}}_{\ast}(x)} - \tau \nabla V_{\varepsilon}(x) &\big> \ud \mu^{n} \ge 0\\
        &\text{ for all $\boldsymbol{\mathrm{u}} \in C^{1}(\mathbb{T}^{d};\mathbb{R}^{d})$.}
    \end{aligned}
\end{equation*}
If one replaces the test vector field $\boldsymbol{\mathrm{u}}$ with $-\boldsymbol{\mathrm{u}}$, the inequality will be reversed (see, for example, \cite[Equation (10.4)]{villani2008optimal}). By this symmetry, we get the first-order optimality condition.
\end{proof}

We next show the following two lemmas for proving the convergence of the $\varepsilon$-JKO scheme.
\begin{lemma}[a priori estimate]\label{p-a-prior}Let $\{\mu^{n}\}$ be the time nodes at $t = n\tau$ defined by the $\varepsilon$-JKO scheme with initial value $\mu_{0}$ such that $S(\mu_{0}) < \infty$. Then there exists a constant $C$ depending only on $\mu_{0}$ such that
\begin{equation*}
    \sum_{n = 1}^{N} W_{\varepsilon}(\mu^{n - 1}, \mu^{n})^{2} \le C\tau.
\end{equation*}
\end{lemma}

\begin{lemma}[consistency]\label{JKO-bound}Given an absolutely continuous probability measure $\mu^{n - 1} \in P(\mathbb{T}^{d})$, and let $\mu^{n} = \arg\min_{\mu} \mathcal{J}_{\varepsilon, \tau}(\mu, \mu^{n - 1})$, we have
\begin{equation*}
    \begin{aligned}
        \bigg|\int_{\bR} \frac{\mu^{n} - \mu^{n - 1}}{\tau} \varphi\ud x + \int_{\bR} \Big<A\Big(\frac{x}{\varepsilon}\Big) \big(\nabla \mu^{n} &+ \nabla V_{\varepsilon}(x) \mu^{n}\big), \nabla \varphi\Big> \ud x\bigg|\\
        &\quad\le \frac{C_{\varepsilon}(\varphi)}{2\tau} W_{\varepsilon}(\mu^{n}, \mu^{n - 1})^{2},
    \end{aligned}
\end{equation*}
for any $\varphi \in C^{2}(\mathbb{T}^{d})$ and a constant $C_{\varepsilon}(\varphi)$ depending on $\varphi$ and $\varepsilon$.
\end{lemma}

Combining these two lemmas, whose proofs will be established later, we prove the convergence of the $\varepsilon$-JKO scheme.

\begin{proof}[Proof of Theorem~\ref{JKO-t}]
Let $\mu_{\tau}(t)$ be the solution of the $\varepsilon$-JKO scheme with step-size $\tau$. According to the compactness of minimizing movement trajectories in general metric spaces \cite[Proposition 2.2.3]{ambrosio2008gradient}, there exists an absolutely continuous curve $\mu(t)$ in $(P(\mathbb{T}^{d}), W_{2})$ such that
\begin{equation*}
    \mu_{\tau}(t) \to \mu(t) \text{ in $W_2$, for all } t \in (0, \infty)
\end{equation*}
as $\tau \to 0$, up to the extraction of a subsequence. To identify the limit $\mu$, we fix a time horizon $T > 0$, and let $N = T/\tau$ and $t^{n} = n\tau$. By Lemmas~\ref{JKO-bound} and~\ref{p-a-prior}, we have
\begin{equation}\label{JKO-bound-ineq}
    \begin{aligned}
        \bigg|\sum_{n = 1}^{N}\bigg[\int_{\bR} \frac{\mu^{n} - \mu^{n - 1}}{\tau} \varphi\ud x + \int_{\bR} \Big<&A\Big(\frac{x}{\varepsilon}\Big) \big(\nabla \mu^{n} + \nabla V_{\varepsilon}(x) \mu^{n}\big), \nabla \varphi\Big> \ud x\bigg]\tau\bigg|\\
        &\le \tau \times \frac{C_{\varepsilon}(\varphi)}{2\tau} \sum_{n = 1}^{N} W_{\varepsilon}(\mu^{n}, \mu^{n - 1})^{2} \le C\tau
    \end{aligned}
\end{equation}
for all $\varphi \in C^{1}([0, \infty), C(\mathbb{T}^{d})) \cap C([0, \infty), C^{2}(\mathbb{T}^{d}))$. Taking the limit as $\tau \to 0$, we have
\begin{equation*}
    \begin{aligned}
        \bigg|\int_{0}^{T}\int_{\mathbb{T}^{d}} -\mu(t)\partial_{t}\varphi + \Big<A\Big(\frac{x}{\varepsilon}\Big) \big(&\nabla \mu(t) + \mu(t)\nabla V_{\varepsilon}\big), \nabla \varphi\Big> \ud x \ud t\\
        &- \int_{\mathbb{T}^{d}} \varphi(0) \mu_{0}\ud x + \int_{\mathbb{T}^{d}} \varphi(T) \mu(T)\ud x\bigg| = 0\,,
    \end{aligned}
\end{equation*}
which means that the limit $\mu(t)$ satisfy~\eqref{0} in the weak sense. Since $\mu(t)$ is uniquely determined by initial value $\mu_{0}$, the whole sequence $\mu_{\tau}$ converges to $\mu$ as $\tau \to 0$.
\end{proof}

The proof of the a priori estimates for minimizing movement schemes is standard; see, for example, \cite[Equation (46)]{jordan1998variational}. We nevertheless include it here for completeness. 
\begin{proof}[Proof of Lemma~\ref{p-a-prior}]
By the construction of minimizing movements, we have
\begin{equation*}
    \sum_{n = 1}^{N} W_{\varepsilon}(\mu^{n - 1}, \mu^{n})^{2} \le 2\tau \big(H_{\varepsilon}(\mu_{0}) - H_{\varepsilon}(\mu^{N})\big).
\end{equation*}
Since $V_{\varepsilon}(x) \ge 0$ by Assumption~\ref{A5} and $x \log(x) \ge -e^{-1}$, we can get
\begin{equation*}
    H_{\varepsilon}(\mu^{N}) = S(\mu^{N}) + \int_{\mathbb{T}^{d}} V_{\varepsilon}(x) \ud \mu^{N} \ge S(\mu^{N}) \ge -e^{-1}.
\end{equation*}
According to Assumption~\ref{A4}, we further have
\begin{equation*}
    H_{\varepsilon}(\mu_{0}) = S(\mu_{0}) + \int_{\mathbb{T}^{d}} V_{\varepsilon}(x) \ud \mu_{0} \le S(\mu_{0}) + \lambda\,,
\end{equation*}
which gives the a priori estimate since $S(\mu_{0}) < \infty$ by assumption.
\end{proof}

The proof of the consistency is based on the following pointwise Taylor expansion estimate of multipliers.
\begin{lemma}\label{lemma2}
For all $\varphi \in C^{2}(\bR)$, there exists a constant $0 < C_{\varepsilon}(\varphi) < \infty$ depending on $\varphi$, such that for all $(x, y) \in \mathbb{T}^{d} \times \mathbb{T}^{d}$, we have
\begin{equation*}
    \Big|\varphi(x) - \varphi(y) + \big<\xi_{x \to y}, \nabla\varphi(x)\big>\Big| \le \frac{C_{\varepsilon}(\varphi)}{2} c_{\varepsilon}(x, y)\,,
\end{equation*}
where $\xi_{x \to y} = \dot{z}(0)$ is the initial velocity of an action minimizing curve $z(t)$ connecting $x$ to $y$.
\end{lemma}
\begin{proof}[Proof of Lemma~\ref{lemma2}]
Consider the representing $\varphi$ using exponential map
\begin{equation*}
    \overline{\varphi} = \varphi \circ \exp_{x} : \mathbb{R}^{d} \to \mathbb{R}.
\end{equation*}
where $\exp_{x}$ is the exponential map of the Lagrangian action $L(x, v) = \la A(x/\varepsilon)v, v\ra$. Since $\exp_{x}$ is smooth, we have $\overline{\varphi} \in C^{2}(\mathbb{R}^{d})$. As $\mathbb{T}^{d}$ is compact, we can find a compact subset $K$ of $\mathbb{R}^{d}$ such that $\exp_{x}(K) = \mathbb{T}^{d}$. Therefore, we only consider the values of $\overline{\varphi}$ defined on $K$. In this regard, we assume that $\overline{\varphi} \in C^{2}(K)$ without loss of generality.

Given any $\xi \in \mathbb{R}^{d}$, we have
\begin{equation*}
    \overline{\varphi}(\xi) - \overline{\varphi}(0) - \big<\xi, \nabla \overline{\varphi}(0)\big> = \int_{0}^{1}(1 - s)\big<\xi, \nabla^{2}\overline{\varphi}(s \xi)\xi\big>\ud s.
\end{equation*}
Since $\overline{\varphi} \in C^{2}(K)$, we can find a constant $C_{\varepsilon}(\overline{\varphi})$ such that $\nabla^{2}\overline{\varphi}(s \xi) \preceq C_{\varepsilon}(\overline{\varphi})I_{d \times d}$, and therefore, by using the fact that $\Lambda_{1}^{-1}I_{d \times d} \preceq A(x)^{-1}$
\begin{equation}\label{lemma2-2}
    \Big|\overline{\varphi}(\xi) - \overline{\varphi}(0) - \big<\xi, \nabla \overline{\varphi}(0)\big>\Big| \le \frac{C_{\varepsilon}(\overline{\varphi})}{2}|\xi|^{2} \le \frac{C_{\varepsilon}(\varphi)}{2}\big<\xi, A(x/\varepsilon)^{-1}\xi\big>.
\end{equation}
for another constant $C_{\varepsilon}(\varphi) = \Lambda_{1} C_{\varepsilon}(\overline{\varphi})$. Finally, let $\xi = \xi_{x \to y}$ be the initial velocity of an action minimizing curve connecting $x$ to $y$. By Lemmas~\ref{velcoti-c} and~\ref{exp}, we have
\begin{equation*}
    \nabla \overline{\varphi}(0) = \nabla \varphi(x) \text{ and } \big<\xi_{x \to y}, A(x/\varepsilon)^{-1}\xi_{x \to y}\big> = c_{\varepsilon}(x, y).
\end{equation*}
By taking these two facts into~\eqref{lemma2-2}, we get the estimate.
\end{proof}

\begin{proof}[Proof of Lemma~\ref{JKO-bound}]
For any $\varphi \in C^{2}(\mathbb{T}^{d})$, we take $\boldsymbol{\mathrm{u}}(x) = A(x/\varepsilon) \nabla \varphi(x)$ in the first-order optimality condition (Proposition~\ref{op-ejko}) to get
\begin{equation}\label{single-weak}
    \int_{\mathbb{T}^{d}} \frac{1}{\tau}\big<\xi_{x \to \boldsymbol{\mathrm{r}}_{\ast}(x)}\mu^{n}, \nabla\varphi\big> \ud x = \int_{\bR} \Big<A\Big(\frac{x}{\varepsilon}\Big) \big(\nabla \mu^{n} + \nabla V_{\varepsilon}(x) \mu^{n}\big), \nabla \varphi\Big> \ud x.
\end{equation}
Since $\mu^{n - 1} = \boldsymbol{\mathrm{r}}_{\ast} \# \mu^{n}$, we have $\int_{\mathbb{T}^{d}} \varphi \ud \mu^{n - 1} = \int_{\mathbb{T}^{d}} \varphi \circ \boldsymbol{\mathrm{r}}_{\ast} \ud \mu^{n}$. Therefore,
\begin{equation*}
    \begin{aligned}
        \int_{\bR} \frac{\mu^{n} - \mu^{n - 1}}{\tau} \varphi\ud x &+ \int_{\bR} \Big<A\Big(\frac{x}{\varepsilon}\Big) \big(\nabla \mu^{n} + \nabla V_{\varepsilon}(x) \mu^{n}\big), \nabla \varphi\Big> \ud x\\
        &\quad = \int_{\bR} \frac{\varphi(x) - \varphi(\boldsymbol{\mathrm{r}}_{\ast}(x)) + \big<\xi_{x \to \boldsymbol{\mathrm{r}}_{\ast}(x)}, \nabla\varphi(x)\big>}{\tau} \ud \mu^{n}.
    \end{aligned}
\end{equation*}
By using the Taylor expansion type estimate in Lemma~\ref{lemma2}, one has
\begin{equation*}
    \big|\varphi(x) - \varphi(y) + \big<\xi_{x \to y}, \nabla\varphi(x)\big>\big| \le \frac{C_{\varepsilon}(\varphi)}{2} c_{\varepsilon}(x, y),
\end{equation*}
with some constant $C_{\varepsilon}(\varphi)$. Therefore,
\begin{equation*}
    \bigg|\int_{\bR} \frac{\varphi(x) - \varphi(\boldsymbol{\mathrm{r}}_{\ast}(x)) + \big<\xi_{x \to \boldsymbol{\mathrm{r}}_{\ast}(x)}, \nabla\varphi(x)\big>}{\tau} \ud \mu^{n}\bigg| \le \frac{C_{\varepsilon}(\varphi)}{2\tau} W_{\varepsilon}(\mu^{n}, \mu^{n - 1})^{2}.
\end{equation*}
\end{proof}

\begin{remark}[non-commutativity between $\tau$ and $\varepsilon$]\label{remark-5}Lemma~\ref{JKO-bound} suggests that the $\varepsilon$-JKO scheme remains comparable to the weighted $L^{2}$ scheme as long as $C_{\varepsilon}(\varphi)$ stays uniformly bounded for every $\varphi$. However, we will show below that, for certain choices of $\varphi$, one has $C_{\varepsilon}(\varphi) \to \infty$.
This divergence appears to be the main mechanism underlying the different homogenization limits of the $\varepsilon$-JKO and weighted $L^{2}$ schemes. It also explains the noncommutativity of the limits $\tau\to 0$ and $\varepsilon\to 0$, as illustrated in the diagram in Section~\ref{sec:dicuss}.

Indeed, let $\overline{\varphi} := \varphi \circ \exp_{x}$ from~\eqref{lemma2-2}, on one-dimension domains, we have
\begin{equation*}
    C_{\varepsilon}(\varphi) = \Lambda_{1} \times \sup_{\xi} \big|\mathrm{d}^{2} \overline{\varphi}(\xi)/\mathrm{d} \xi^{2}\big|
\end{equation*}
where
\begin{equation*}
    \frac{\mathrm{d}^{2} \overline{\varphi}(\xi)}{\mathrm{d} \xi^{2}} = \frac{\mathrm{d}^{2} \exp_{x}(\xi)}{\mathrm{d} \xi^{2}} \times \varphi^{\prime}(\exp_{x}(\xi)) + \Big(\frac{\mathrm{d} \exp_{x}(\xi)}{\mathrm{d} \xi}\Big)^{2} \times \varphi^{\prime\prime}(\exp_{x}(\xi)).
\end{equation*}
Here the exponential map of the action $\int_0^1 \langle A(z/\varepsilon)^{-1}\dot{z}, \dot{z} \rangle \ud t$ can be explicitly calculated as follows (when $|\xi|$ is small enough):
\begin{equation*}
    \exp_{x}(\xi) = \sigma^{-1}\big(A(x/\varepsilon)^{-\frac{1}{2}}\xi + \sigma(x)\big) \text{ where } \sigma(z) := \int_{0}^{z} A(t/\varepsilon)^{-\frac{1}{2}} \ud t.
\end{equation*}
In this case, $\mathrm{d}_{\xi}\exp_{x}(\xi)$ remains bounded but $\mathrm{d}_{\xi}^{2}\exp_{x}(\xi) \to \infty$ as $\varepsilon \to 0$, which implies
\begin{equation*}
    C_{\varepsilon}(\varphi) \to \infty \text{ as } \varepsilon \to 0 \text{ for non-constant multipliers $\varphi$.}
\end{equation*}
This blow-up also holds true for high-dimensional cases, which involves the computation of the Ricci curvature of the metric tensor $g_{\varepsilon, x}(\xi, \eta) := \la A(x/\varepsilon)\xi, \eta\ra$.
\end{remark}

\subsubsection{Homogenized JKO scheme}\label{sec:homo-jko}
In this subsection, we identify the time-continuous limit of the homogenized JKO solution $\overline{\mu}_{\tau}$ for $\tau \to 0$. While the main framework for proving the convergence remains the same as in the previous section, the main technical difficulty is that one needs to deal with the nonlinearity introduced by $A_{\mathrm{hom}}$ in the transport distance. In general, one would need to investigate the Finsler structure of the transport metric as explained by \cite[Equation (5.19)]{gao2023homogenization}. To address this problem, we leverage the convexity and homogeneity of $A_{\mathrm{hom}}$ given in Lemma~\ref{limit-B}, to reduce the minimizing geodesics of $\int_{0}^{1} A_{\mathrm{hom}}(\dot{z}) \ud t$ into Euclidean geodesics~\eqref{Jensen}, which substantially simplifies the geometric structure.

\begin{lemma}\label{limit-B}We have the following properties of the action function $A_{\mathrm{hom}}$
\begin{enumerate}
    \item $A_{\mathrm{hom}}$ is convex
    \item $A_{\mathrm{hom}}$ is homogeneous of degree $2$, i.e., $A_{\mathrm{hom}}(\alpha\xi) = \alpha^{2}A_{\mathrm{hom}}(\xi)$.
\end{enumerate}
\end{lemma}
\begin{proof}
The convexity of $A_{\mathrm{hom}}$ is given in \cite[Theorem 5.1]{weinan1991class}. The homogeneity of degree $2$ is given in \cite[Section 5.2]{gao2023homogenization}.
\end{proof}

In the following, we assume $A_{\mathrm{hom}}$ to be uniformly strongly convex.

As $A_{\mathrm{hom}}$ is homogeneous of degree $2$, we have
\begin{equation}\label{homo-1}
    \nabla A_{\mathrm{hom}}(\alpha\xi) = \alpha \nabla A_{\mathrm{hom}}(\xi).
\end{equation}
Furthermore, by the uniform convexity, its gradient map $\nabla A_{\mathrm{hom}}$ satisfies
\begin{equation*}
    \underline{\Lambda}_{0} \|\eta\| \le \|\tfrac{1}{2}\nabla A_{\mathrm{hom}}(\xi + \eta) - \tfrac{1}{2}\nabla A_{\mathrm{hom}}(\xi)\| \le \underline{\Lambda}_{1} \|\eta\|
\end{equation*}
with positive constants $\underline{\Lambda}_{0}, \underline{\Lambda}_{1}$. Let $\tfrac{1}{2}A_{\mathrm{hom}}^{\ast}$ be the Legendre transform of $\tfrac{1}{2}A_{\mathrm{hom}}$, one has $A_{\mathrm{hom}}^{\ast}$ is uniformly strongly convex. By combining the homogeneity of $A_{\mathrm{hom}}$, one can get $\tfrac{1}{2}\nabla A_{\mathrm{hom}}^{\ast} = (\tfrac{1}{2} \nabla A_{\mathrm{hom}})^{-1}$, which means $(\nabla A_{\mathrm{hom}})^{-1}$ is Lipschitz continuous:
\begin{equation}\label{homo-2}
    \big\|(\tfrac{1}{2}\nabla A_{\mathrm{hom}})^{-1}(\xi + \eta) - (\tfrac{1}{2}\nabla A_{\mathrm{hom}})^{-1}(\xi)\big\| \le \underline{\Lambda}_{0}^{-1} \|\eta\|
\end{equation}

Based on the above discussions, we show the following optimality condition of the homogenized JKO scheme.
\begin{theorem}\label{op-hjko}Given $\overline{\mu}^{n - 1} \in P(\mathbb{T}^{d})$, let $\overline{\mu}^{n} = \arg\min_{\mu} \mathcal{J}_{\ast, \tau}(\mu, \overline{\mu}^{n - 1})$, 
and $\boldsymbol{\mathrm{r}}_{\ast}$ be the optimal transport map from $\overline{\mu}^{n}$ to $\overline{\mu}^{n - 1}$ under the transport cost $c_{\mathrm{hom}}(x, y) := \min_{\mathscr{C}(x, y)} \int_{0}^{1} A_{\mathrm{hom}}(\dot{z}) \ud t$ in~\eqref{eps-c-copy}, we have
\begin{equation*}
    \begin{aligned}
        -\int_{\mathbb{T}^{d}} \nabla \cdot \boldsymbol{\mathrm{u}} \ud \overline{\mu}^{n} = \int_{\mathbb{T}^{d}} \frac{1}{\tau} \big<\boldsymbol{\mathrm{u}}(x), \tfrac{1}{2}\nabla A_{\mathrm{hom}}(\xi_{x \to \boldsymbol{\mathrm{r}}_{\ast}(x)}) &- \tau \nabla V(x)\big> \ud \overline{\mu}^{n}\\
        &\text{ for all } \boldsymbol{\mathrm{u}} \in C^{1}(\mathbb{T}^{d}; \mathbb{R}^{d}),
    \end{aligned}
\end{equation*}
where $\xi_{x \to y} = \dot{z}(0)$ is the initial velocity of an action minimizing curve $z(t)$ from $x$ to $y$ under the action $\int_{0}^{1}A_{\mathrm{hom}}(\dot{z}) \ud t$.
\end{theorem}
\begin{corollary}[Regularity]Let $\overline{\mu}^{n} = \arg\min_{\mu} \mathcal{J}_{\ast, \tau}(\mu, \overline{\mu}^{n - 1})$, we have $\overline{\mu}^{n} \in P(\mathbb{T}^{d}) \cap W^{1, 1}(\mathbb{T}^{d})$.
\end{corollary}
\begin{proof}
Since $\boldsymbol{\mathrm{r}}_{\ast} : \mathbb{T}^{d} \to \mathbb{T}^{d}$ is an optimal transport map with value in compact sets and $\xi$ is the initial velocity vector of action minimizing curves, by Lemma~\ref{velcoti-c}, we have
\begin{equation*}
    \Lambda_{0} |\xi_{x \to y}|^{2} \le A_{\mathrm{hom}}(\xi_{x \to \boldsymbol{\mathrm{r}}_{\ast}(x)}) = c_{\mathrm{hom}}(x, \boldsymbol{\mathrm{r}}_{\ast}(x)),
\end{equation*}
which means that $\xi_{x \to \boldsymbol{\mathrm{r}}_{\ast}(x)}$ is bounded. Due to Assumption~\ref{A4}, $\nabla V(x)$ is also bounded. Since $H(\overline{\mu}^{n}) < \infty$, $\overline{\mu}^{n}$ is absolutely continuous with density in $L^{1}(\mathbb{T}^{d})$, and $\overline{\mu}^{n}$ has a distributional gradient in $L^{1}(\mathbb{T}^{d})$ by Theorem~\ref{op-hjko}.
\end{proof}

\begin{proof}[Proof of Theorem~\ref{op-hjko}]
The proof of the first-order optimality condition is essentially the same as that of Proposition~\ref{op-ejko}. We let $\boldsymbol{\mathrm{u}} : \mathbb{T}^{d} \to \mathbb{R}^{d}$ be a $C^{1}$ vector field, and $\boldsymbol{\mathrm{r}}(\cdot, s)$ be the flow induced by $\boldsymbol{\mathrm{u}}$ as that in~\eqref{flow-def}.

Let $\mu(s) = \boldsymbol{\mathrm{r}}(\cdot, s) \# \overline{\mu}^{n}$. By the optimality of $\overline{\mu}^{n}$, we have
\begin{equation}\label{limit-j-hom}
    \begin{aligned}
        &\frac{\mathcal{J}_{\ast, \tau}(\mu(s), \overline{\mu}^{n - 1}) - \mathcal{J}_{\ast, \tau}(\overline{\mu}^{n}, \overline{\mu}^{n - 1})}{s}\\
        &= \frac{H(\mu(s), \overline{\mu}^{n - 1}) - H(\overline{\mu}^{n}, \overline{\mu}^{n - 1})}{s} + \frac{W_{\mathrm{GH}}(\mu(s), \overline{\mu}^{n - 1}) - W_{\mathrm{GH}}(\overline{\mu}^{n}, \overline{\mu}^{n - 1})}{2\tau \times s} \ge 0.
    \end{aligned}
\end{equation}
For the entropy functional, we have
\begin{equation}\label{limit-h-hom}
    \lim_{s \to 0}\frac{H(\mu(s)) - H(\overline{\mu}^{n})}{s} = \int_{\mathbb{T}^{d}} \la\nabla V, \boldsymbol{\mathrm{u}} \ra - \nabla \cdot \boldsymbol{\mathrm{u}} \ud \overline{\mu}^{n}.
\end{equation}
To differentiate the distance $W_{\mathrm{GH}}$, we let $\omega \in \Pi(\overline{\mu}^{n}, \overline{\mu}^{n - 1})$ be the optimal coupling of the distance $W_{\mathrm{GH}}(\overline{\mu}^{n}, \overline{\mu}^{n - 1})$. Here $\boldsymbol{\mathrm{r}}(\cdot, s)$ induces a time dependent coupling $\omega[s] \in \Pi(\mu(s), \overline{\mu}^{n - 1})$ by pushing-forward the first marginal measure:
\begin{equation*}
    \int_{\mathbb{T}^{d} \times \mathbb{T}^{d}} \varphi(x, y) \ud \omega[s](x, y) = \int_{\mathbb{T}^{d} \times \mathbb{T}^{d}} \varphi(\boldsymbol{\mathrm{r}}(x, s), y)\ud \omega(x, y), \text{ for all } \varphi \in C(\mathbb{T}^{d}).
\end{equation*}
By using the fact that $W_{\mathrm{GH}}(\mu(s), \overline{\mu}^{n - 1})^{2} \le \int_{\mathbb{T}^{d} \times \mathbb{T}^{d}} c_{\mathrm{hom}}(x, y) \ud \omega[s]$, we have
\begin{equation*}
    \begin{aligned}
        \frac{1}{2s} \Big(W_{\mathrm{GH}}(\mu(s), \overline{\mu}^{n - 1})^{2} &- W_{\mathrm{GH}}(\overline{\mu}^{n}, \overline{\mu}^{n - 1})^{2}\Big)\\
        &\le \int_{\mathbb{T}^{d} \times \mathbb{T}^{d}} \frac{1}{2s}\big(c_{\mathrm{hom}}(\boldsymbol{\mathrm{r}}(x, s), y) - c_{\mathrm{hom}}(x, y)\big) \ud \omega(x, y).
    \end{aligned} 
\end{equation*}
Furthermore, by \cite[Proposition 10.15]{villani2008optimal}, the cost $c_{\mathrm{hom}}(x, y)$ given by the Lagrangian action $\int_{0}^{1}A_{\mathrm{hom}}(\dot{z}) \ud t$ is super-differentiable:
\begin{equation*}
    \begin{aligned}
        \frac{c_{\mathrm{hom}}(\boldsymbol{\mathrm{r}}(x, s), y) - c_{\mathrm{hom}}(x, y)}{2s} &\le \frac{\big<\boldsymbol{\mathrm{r}}(x, s) - x, p\big>}{2s} + \frac{o(|\boldsymbol{\mathrm{r}}(x, s) - x|)}{2s}\\
        & \text{ with super-gradient } p = -\nabla A_{\mathrm{hom}}(\xi_{x \to y})
    \end{aligned}
\end{equation*}
where $\xi_{x \to y}$ is initial velocity vector of an action minimizing curve of $\int_{0}^{1} A_{\mathrm{hom}}(\dot{z}(t))\ud t$ connecting $x$ to $y$. Therefore,
\begin{equation}\label{limit-w-hom}
    \begin{aligned}
        \limsup_{s \to 0} \frac{1}{2s}\Big(W_{\mathrm{GH}}(\mu(s), \overline{\mu}^{n - 1})^{2} &- W_{\mathrm{GH}}(\overline{\mu}^{n}, \overline{\mu}^{n - 1})^{2}\Big)\\
        &\le \int_{\mathbb{T}^{d} \times \mathbb{T}^{d}} -\big<\boldsymbol{\mathrm{u}}(x), \tfrac{1}{2}\nabla A_{\mathrm{hom}}(\xi_{x \to y})\big> \ud \omega(x, y).
    \end{aligned}
\end{equation}

Combining~\eqref{limit-h-hom} and~\eqref{limit-w-hom}, we take the $\limsup_{s \to 0}$ in~\eqref{limit-j-hom} to get,
\begin{equation*}
    \int_{\mathbb{T}^{d}} \la\nabla V, \boldsymbol{\mathrm{u}}\ra - \nabla \cdot \boldsymbol{\mathrm{u}} \ud \overline{\mu}^{n} - \frac{1}{\tau} \int_{\mathbb{T}^{d} \times \mathbb{T}^{d}}\big<\boldsymbol{\mathrm{u}}(x), \tfrac{1}{2}\nabla A_{\mathrm{hom}}(\xi_{x \to y})\big> \ud \omega(x, y) \ge 0.
\end{equation*}
Let $\boldsymbol{\mathrm{r}}_{\ast}$ be the optimal transport map from $\overline{\mu}^{n}$ to $\overline{\mu}^{n - 1}$, since $\omega = (\mathrm{id}, \boldsymbol{\mathrm{r}}_{\ast}) \# \overline{\mu}^{n}$, we can further get
\begin{equation*}
    \begin{aligned}
        \int_{\mathbb{T}^{d}} -\nabla \cdot \boldsymbol{\mathrm{u}} - \frac{1}{\tau} \big<\boldsymbol{\mathrm{u}}(x), \tfrac{1}{2}\nabla A_{\mathrm{hom}}(\xi_{x \to \boldsymbol{\mathrm{r}}_{\ast}(x)}) - \tau \nabla V(x) &\big> \ud \overline{\mu}^{n} \ge 0.\\
        &\text{ for all } \boldsymbol{\mathrm{u}} \in C^{1}(\mathbb{T}^{d}, \mathbb{R}^{d}).
    \end{aligned}
\end{equation*}
Finally, if we replace $\boldsymbol{\mathrm{u}}$ with $- \boldsymbol{\mathrm{u}}$, the above inequality will be reversed. By using this symmetry, we get the first-order optimality condition.
\end{proof}

Similar to the convergence of $\varepsilon$-JKO scheme in the above subsection, we establish the following transport estimate of consistency.
\begin{lemma}[transport estimate of consistency]\label{homo-convergence-t}Given an absolutely continuous probability measure $\overline{\mu}^{n - 1} \in P(\mathbb{T}^{d})$, and let $\overline{\mu}^{n} = \arg\min_{\mu} \mathcal{J}_{\ast, \tau}(\mu, \overline{\mu}^{n - 1})$, we have
\begin{equation}
    \begin{aligned}
        \bigg|\int_{\mathbb{T}^{d}} \frac{\overline{\mu}^{n} - \overline{\mu}^{n - 1}}{\tau} \varphi \ud x + \int_{\mathbb{T}^{d}} \nabla \varphi \cdot \big(\tfrac{1}{2}\nabla A_{\mathrm{hom}}\big)^{-1}\big(\nabla\overline{\mu}^{n} &+ \overline{\mu}^{n} \nabla V\big) \ud x\bigg|\\
        &\le \frac{C(\varphi)}{2\tau} W_{\mathrm{GH}}(\overline{\mu}^{n}, \overline{\mu}^{n - 1})^{2},
    \end{aligned}
\end{equation}
for all $\varphi \in C^{2}(\mathbb{T}^{d})$, where $C(\varphi)$ is a constant depending on $\varphi$.
\end{lemma}

Combining with the a-prior estimate $\sum_{n = 1}^{N}W_{\mathrm{GH}}(\overline{\mu}^{n}, \overline{\mu}^{n - 1})^{2} \le C\tau$ the same as that in Lemma~\ref{p-a-prior}, we can derive the time continuous limit of the homogenized JKO scheme.
\begin{proof}[Proof of Theorem~\ref{prop-homo-jko}]
Let $\overline{\mu}_{\tau}(t)$ be the solution of the homogenized JKO scheme with step-size $\tau$. According to the compactness of minimizing movement trajectories in general metric spaces \cite[Proposition 2.2.3]{ambrosio2008gradient}, there exists an absolutely continuous curve $\overline{\mu}(t)$ in $(P(\mathbb{T}^{d}), W_{2})$ such that
\begin{equation*}
    \overline{\mu}_{\tau}(t) \to \overline{\mu}(t) \text{ in $W_2$, for all } t \in (0, \infty)
\end{equation*}
as $\tau \to 0$, up to the extraction of a subsequence. To identify the limit $\overline{\mu}$, we fix a time horizon $T > 0$, and let $N = T/\tau$ and $t^{n} = n\tau$ According to Lemma~\ref{homo-convergence-t}, we have
\begin{equation}\label{homo-3}
    \begin{aligned}
        \bigg|\sum_{n = 1}^{N}\bigg[\int_{\bR} \frac{\overline{\mu}^{n} - \overline{\mu}^{n - 1}}{\tau} \varphi\ud x + \int_{\bR} \nabla \varphi \cdot \big(&\tfrac{1}{2}\nabla A_{\mathrm{hom}}\big)^{-1}\big(\nabla\overline{\mu}^{n} + \overline{\mu}^{n} \nabla V\big) \ud x\bigg]\tau\bigg|\\
        &\quad \le \tau \times \frac{C(\varphi)}{2\tau} \sum_{n = 1}^{N} W_{\varepsilon}(\overline{\mu}^{n}, \overline{\mu}^{n - 1})^{2} \le C\tau,
    \end{aligned}
\end{equation}
for all $\varphi \in C^{1}([0, \infty), C(\mathbb{T}^{d})) \cap C([0, \infty), C^{2}(\mathbb{T}^{d}))$.

To get the limit $\tau \to 0$, one needs to pass to the limit of $\nabla \overline{\mu}_{\tau}$ inside the nonlinear function $(\tfrac{1}{2}\nabla A_{\mathrm{hom}})^{-1}$. However, the convergence $\overline{\mu}_{\tau}(t) \to \overline{\mu}(t)$ in $W_2$, obtained from compactness of minimizing movement trajectories, is not sufficient for this purpose. This is precisely where we use the additional regularity assumption that $\overline{\mu}_{\tau}(t) \to \overline{\mu}(t)$ in $W^{1, 1}(\mathbb{T}^{d})$, together with the uniform convexity of $\nabla A_{\mathrm{hom}}$. In particular, from~\eqref{homo-2},
\begin{equation*}
    \begin{aligned}
        \Big|(\nabla A_{\mathrm{hom}})^{-1}\big(\nabla \overline{\mu}_{\tau} + \overline{\mu}_{\tau} \nabla V(x)\big) &- (\nabla A_{\mathrm{hom}})^{-1}\big(\nabla \overline{\mu} + \overline{\mu}\nabla V(x)\big)\Big|\\
        &\le \underline{\Lambda}_{0}^{-1} \Big|\big(\nabla \overline{\mu}_{\tau} + \overline{\mu}_{\tau} \nabla V(x)\big) - \big(\nabla \overline{\mu} + \overline{\mu}\nabla V(x)\big)\Big|,
    \end{aligned}
\end{equation*}
and we can therefore take the limit $\tau \to 0$ in~\eqref{homo-3} to get:
\begin{equation*}
    \begin{aligned}
        \bigg|\int_{0}^{T}\int_{\mathbb{T}^{d}} -\overline{\mu}(t)\partial_{t}\varphi + \nabla \varphi \cdot \big(&\tfrac{1}{2}\nabla A_{\mathrm{hom}}\big)^{-1}\big(\nabla \overline{\mu} + \overline{\mu}\nabla V(x)\big) \ud x \ud t\\
        &\quad - \int_{\mathbb{T}^{d}} \varphi(0) \mu_{0}\ud x + \int_{\mathbb{T}^{d}} \varphi(T) \overline{\mu}(T)\ud x\bigg| = 0.
    \end{aligned}
\end{equation*}
Hence $\overline{\mu}(t)$ solves the following quasilinear parabolic equations in divergence form.
\begin{equation*}
    \frac{\partial \overline{\mu}}{\partial t} - \nabla \cdot \Big[\big(\tfrac{1}{2}\nabla A_{\mathrm{hom}}\big)^{-1}\big(\nabla \overline{\mu} + \overline{\mu}\nabla V(x)\big)\Big] = 0,
\end{equation*}
with initial value $\mu_{0}$.
\end{proof}

To prove Lemma~\ref{homo-convergence-t}, we establish the following point-wise Taylor-expansion type estimate.
\begin{lemma}\label{limiting-taylor} Let $\xi_{x \to y} = \dot{z}(0)$ be the initial velocity vector of an action minimizing curve $z(t)$ connecting $x$ to $y$ under the action $\int_{0}^{1} A_{\mathrm{hom}}(\dot{z}(t))\ud t$, we have
\begin{equation*}
    \Big|\varphi(y) - \varphi(x) - \big<\xi_{x \to y}, \nabla \varphi(x)\big>\Big| \le \frac{C(\varphi)}{2} c_{\mathrm{hom}}(x, y),
\end{equation*}
for all $\varphi \in C^{2}(\mathbb{T}^{d})$, with some constant $C(\varphi)$ depending on $\varphi$, and the transport cost $c_{\mathrm{hom}}(x, y) := \min_{\mathscr{C}(x, y)} \int_{0}^{1} A_{\mathrm{hom}}(\dot{z}) \ud t$ in~\eqref{eps-c-copy}.
\end{lemma}
\begin{proof}[Proof of Lemma~\ref{limiting-taylor}]
By leveraging the convexity of $A_{\mathrm{hom}}$ via Jensen's inequality \cite[Example 10.21]{villani2008optimal}, one has
\begin{equation}\label{Jensen}
    A_{\mathrm{hom}}\bigg(\int_{0}^{1} \dot{z}(t) \ud t\bigg) \le \int_{0}^{1} A_{\mathrm{hom}}(\dot{z}(t))\ud t \text{ for } z \in H^{1}([0, 1], \mathbb{T}^{d}),
\end{equation}
with $z(0) = x$ and $z(1) = y$. The equality in~\eqref{Jensen} holds only if the velocity $\dot{z}(t)$ is a constant vector. Hence, by taking minimal on the RHS of~\eqref{Jensen}, there exists $\xi_{x \to y}$ such that
\begin{equation}\label{char-c-ast}
    c_{\mathrm{hom}}(x, y) = A_{\mathrm{hom}}(\xi_{x \to y}).
\end{equation}
Let $\exp$ be the exponential map of the action $L(x, v) = |v|^{2}$, we have $\exp_{x}(\xi_{x \to y}) = y$. Similar to Lemma~\ref{lemma2}, we composite test function $\varphi$ with this exponential map $\exp_{x}$ to get the following point-wise Taylor-expansion type estimate
\begin{equation*}
    \Big|\varphi(y) - \varphi(x) - \big<\xi_{x \to y}, \nabla \varphi(x)\big>\Big| \le \sup_{z \in \mathbb{T}^{d}}\frac{|\nabla^{2}\varphi(z)|}{2} |\xi_{x \to y}|^{2}.
\end{equation*}
Since $\varphi \in C^{2}(\mathbb{T}^{d})$ and $\Lambda_{0}|\xi|^{2} \le A_{\mathrm{hom}}(\xi)$, we have $\sup_{z \in \mathbb{T}^{d}}|\nabla^{2}\varphi(z)| < \infty$ and we can find a constant $C(\varphi)$ such that
\begin{equation*}
    \Big|\varphi(y) - \varphi(x) - \big<\xi_{x \to y}, \nabla \varphi(x)\big>\Big| \le \frac{C(\varphi)}{2} A_{\mathrm{hom}}(\xi_{x \to y}).
\end{equation*}
Since $c_{\mathrm{hom}}(x, y) = A_{\mathrm{hom}}(\xi_{x \to y})$ by~\eqref{char-c-ast}, we get the estimate.
\end{proof}

\begin{proof}[Proof of Lemma~\ref{homo-convergence-t}]
By the optimality condition (Theorem~\ref{op-hjko}), we have
\begin{equation*}
    \tau\nabla \psi[\overline{\mu}^{n}](x) = \big(\tfrac{1}{2}\nabla A_{\mathrm{hom}}\big)(\xi_{x \to \boldsymbol{\mathrm{r}}_{\ast}(x)}) \text{ for $\overline{\mu}^{n}$ a.e., with $\psi[\mu] = \log(\mu) + V$ }
\end{equation*}
which is well-defined on $\{x : \overline{\mu}^{n}(x) > 0\}$. Taking into account of the one-homogeneity of $\nabla A_{\mathrm{hom}}$ in~\eqref{homo-1}, we have
\begin{equation*}
    \xi_{x \to \boldsymbol{\mathrm{r}}_{\ast}(x)} = \tau \big(\tfrac{1}{2}\nabla A_{\mathrm{hom}}\big)^{-1}\big(\nabla \psi[\overline{\mu}^{n}](x)\big), \text{ for $\overline{\mu}^{n}$ a.e.}
\end{equation*}
Because $\nabla A_{\mathrm{hom}}$ is one-homogeneous, we can get
\begin{equation*}
    \begin{aligned}
        \int_{\mathbb{T}^{d}} \frac{\overline{\mu}^{n} - \overline{\mu}^{n - 1}}{\tau} \varphi \ud x &+ \int_{\mathbb{T}^{d}} \nabla \varphi \cdot \big(\tfrac{1}{2}\nabla A_{\mathrm{hom}}\big)^{-1}\big(\nabla\overline{\mu}^{n} + \overline{\mu}^{n} \nabla V\big) \ud x\\
        &= \int_{\mathbb{T}^{d}} \frac{\overline{\mu}^{n} - \overline{\mu}^{n - 1}}{\tau} \varphi \ud x + \int_{\mathbb{T}^{d}} \nabla \varphi \cdot \big(\tfrac{1}{2}\nabla A_{\mathrm{hom}}\big)^{-1}\big(\nabla \psi[\overline{\mu}^{n}](x)\big) \ud \overline{\mu}^{n}\\
        &= \int_{\bR} \frac{\varphi(x) - \varphi(\boldsymbol{\mathrm{r}}_{\ast}(x)) + \big<\xi_{x \to \boldsymbol{\mathrm{r}}_{\ast}(x)}, \nabla\varphi(x)\big>}{\tau} \ud \overline{\mu}^{n}.
    \end{aligned}
\end{equation*}
By using the Taylor expansion type estimation in Lemma~\ref{limiting-taylor},
\begin{equation*}
    \big|\varphi(x) - \varphi(\boldsymbol{\mathrm{r}}_{\ast}(x)) + \big<\xi_{x \to \boldsymbol{\mathrm{r}}_{\ast}(x)}, \nabla\varphi(x)\big>\big| \le \frac{C(\varphi)}{2} c_{\mathrm{hom}}(x, \boldsymbol{\mathrm{r}}_{\ast}(x)),
\end{equation*}
with some constant $C(\varphi)$. Hence, we have
\begin{equation*}
    \begin{aligned}
        \bigg|\int_{\mathbb{T}^{d}} \frac{\overline{\mu}^{n} - \overline{\mu}^{n - 1}}{\tau} \varphi \ud x + \int_{\mathbb{T}^{d}} \nabla \varphi \cdot \big(\tfrac{1}{2}\nabla A_{\mathrm{hom}}\big)^{-1}\big(\nabla\overline{\mu}^{n} &+ \overline{\mu}^{n} \nabla V\big) \ud x\bigg|\\
        &\le \frac{C(\varphi)}{2\tau} W_{\mathrm{GH}}(\overline{\mu}^{n}, \overline{\mu}^{n - 1})^{2}
    \end{aligned}
\end{equation*}
which is the transport estimate of consistency.
\end{proof}

\section{Conclusion}\label{sec4}
In this work, we analyzed two minimizing movement schemes for $\varepsilon$-Fokker-Planck equations with rapidly changing diffusion constants. We studied their homogenization limit, and derived the time-continuous limit of the homogenized minimizing movement schemes. From our analysis, we have shown that the weighted $L^{2}$ scheme captures the correct time-continuous homogenization limit constructed in \cite{gao2023homogenization}, while the JKO scheme fails to be asymptotic-preserving.

Deriving the asymptotic homogenization of aggregation-diffusion equations would be important to understand the dynamics of general Wasserstein gradient flows on inhomogeneous media, where aggregation-diffusion equations on inhomogeneous media were considered by \cite{bedrossian2011inhomogeneous, bedrossian2013global}. From our results, it is direct to prove that JKO scheme fails to capture the homogenization limit in general. It remains to exam if the backward Euler scheme still preserves the time-continuous asymptotic homogenization.

\section*{Acknowledgments}
The first author would like to thank Professor Jos\'e A. Carrillo for his insightful discussion.

\appendix
\section{Numerical Results}\label{appendix a}We explain the model equation in Figure~\ref{fig:1} in detail. In particular,
\begin{equation*}
    \partial_{t} \mu - \nabla \cdot\Big[A\Big(\frac{x}{\varepsilon}\Big)\nabla \mu\Big] = 0 \text{ where } A(x) = (1 + \tfrac{1}{2}\sin(2\pi x))^{2}
\end{equation*}
on $\mathbb{R}$, with initial value 
\begin{equation*}
    \mu_{0}(x) = (2\pi)^{-\frac{1}{2}} \exp(-x^{2}/2).
\end{equation*}
Since the solution decays exponentially fast as $x \to \infty$, we consider the computational domain to be $\Omega = [-6, 6]$. We set the scale parameter to $\varepsilon = 10^{-5}$, and choose time step size $\tau = 0.025$ and spatial grid $h = 0.0025$.

\subsection{Numerical methods}
In the weighted $L^{2}$ scheme, instead of solving the minimizing movement formulation, we obtain $\mu^n$ directly from the backward Euler scheme, namely,
\begin{equation*}
\frac{\mu^{n} - \mu^{n - 1}}{\tau} - \nabla \cdot\Big[A\Big(\frac{x}{\varepsilon}\Big)\nabla \mu^{n}\Big] = 0.
\end{equation*}
This equation is the optimality condition of the weighted $L^2$ scheme, discretized in space by central differences.

In the $\varepsilon$-JKO scheme, we use the back-and-forth method \cite{jacobs2020fast, jacobs2021back} and adopt the Nash-Kuiper isometric embedding \cite{liu2025variational} to compute
\begin{equation*}
    c_{\varepsilon}(x, y) = |\sigma(x) - \sigma(y)|^{2} \text{ where } \sigma(x) := \int_{0}^{x} A(s/\varepsilon)^{-\frac{1}{2}} \ud s,
\end{equation*}
where $\sigma$ can be computed analytically for the considered $A(x)$. In order to get $\mu^{n}$, we iteratively update the potential pair $(\phi, \psi)$ via the following steps until they converge:
\begin{equation*}
    \begin{cases}
        \phi_{(\ell + \frac{1}{2})} &= \phi_{(\ell)} + \nabla_{H} J(\phi_{(\ell)})\\
        \psi_{(\ell + \frac{1}{2})} &= \text{ct}(\phi_{(\ell + \frac{1}{2})})\\
        \psi_{(\ell + 1)} &= \psi_{(\ell + \frac{1}{2})} + \nabla_{H} I(\psi_{(\ell + \frac{1}{2})})\\
        \phi_{(\ell + 1)} &= -\text{ct}(-\psi_{(\ell + 1)})
    \end{cases}
\end{equation*}
and then set $\mu^{n} := \exp(\phi_{\infty})$.  
More particularly, 
\begin{itemize}
    \item[1.]$\text{ct}(\phi)$ denotes the $c$-transform defined by
\begin{equation*}
    \text{ct}(\phi)(x) := \min_{y} \phi(y) + \frac{1}{2\tau} |\sigma(x) - \sigma(y)|^{2},
\end{equation*}
which can be evaluated point-wisely on the grid of the discretization.
\item[2.]
 $\nabla_{H}$ denotes the $H^{1}$ gradient operator
\begin{equation*}
    \begin{aligned}
        \nabla_{H} J(\phi) &:= (\Theta_{0}I - \Theta_{1} \Delta)^{-1} \Big[S_{\phi}\# \mu^{n - 1} - \delta U^{\ast}(\phi)\Big]\,;\\
        \nabla_{H} I(\psi) &:= (\Theta_{0}I - \Theta_{1} \Delta)^{-1} \Big[\mu^{n - 1} - S_{-\psi}\#\big(\delta U^{\ast}(-\text{ct}(-\psi))\big)\Big]\,,
    \end{aligned}
\end{equation*}
where $U^{\ast}(\phi) = \int_{\mathbb{T}^{d}}\exp(\phi) \ud x$, and
\begin{equation*}
    \delta U^{\ast}(\phi) = \exp(\phi).
\end{equation*}
Here $\Theta_{0}$ and $\Theta_{1}$ are two parameters of the algorithm. A criteria for parameter tuning is provided in \cite[Theorem 3.3]{jacobs2021back}.

\item[3.]
$S_{\phi} \# \mu$ is the push-forward under the potential $\phi$ given by
\begin{equation*}
    (S_{\phi}\# \mu)(x) := \mu\big(S_{\phi}^{-1}(x)\big) \big|\det\big(DS_{\phi}^{-1}(x)\big)\big|,
\end{equation*}
where $DS$ is the Jacobian of $S$ (and we compute $D$ using a center difference scheme), and the inverse map can be computed as follows,
\begin{equation*}
    S_{\phi}^{-1}(x) := \sigma^{-1}\big(\sigma(x) + \tau \sigma^{\prime}(x)^{-1} D \phi(x)\big).
\end{equation*}
Here $\sigma^{-1}$ can be evaluated using a few Newton iterates since $\sigma$ is monotone.
\end{itemize}

\bibliographystyle{siamplain}

\begin{thebibliography}{10}

\bibitem{allaire2007homogenization}
{\sc G.~Allaire and R.~Orive}, {\em Homogenization of periodic non self-adjoint problems with large drift and potential}, ESAIM Control Optim. Calc. Var., 13 (2007), pp.~735--749.

\bibitem{alouges2016introduction}
{\sc F.~Alouges}, {\em Introduction to periodic homogenization}, Interdiscip. Inf. Sci., 22 (2016), pp.~147--186.

\bibitem{ambrosio2008gradient}
{\sc L.~Ambrosio, N.~Gigli, and G.~Savar{\'e}}, {\em Gradient flows: in metric spaces and in the space of probability measures}, Springer Science \& Business Media, 2008.

\bibitem{bedrossian2013global}
{\sc J.~Bedrossian and I.~C. Kim}, {\em Global existence and finite time blow-up for critical {P}atlak--{K}eller--{S}egel models with inhomogeneous diffusion}, SIAM J. Math. Anal., 45 (2013), pp.~934--964.

\bibitem{bedrossian2011inhomogeneous}
{\sc J.~Bedrossian and N.~Rodr{\'\i}guez}, {\em Inhomogeneous {P}atlak-{K}eller-{S}egel models and aggregation equations with nonlinear diffusion in {R{d}}}, arXiv preprint arXiv:1108.5167,  (2011).

\bibitem{bensoussan2011asymptotic}
{\sc A.~Bensoussan}, {\em Asymptotic analysis for periodic structures, vol. 374}, American Mathematical Soc.,  (2011).

\bibitem{braides2002gamma}
{\sc A.~Braides}, {\em Gamma-convergence for Beginners}, vol.~22, Clarendon Press, 2002.

\bibitem{braides2014local}
{\sc A.~Braides}, {\em Local minimization, variational evolution and $\Gamma$-convergence}, vol.~2094, Springer, 2014.

\bibitem{braides2002riemannian}
{\sc A.~Braides, G.~Buttazzo, and I.~Fragala}, {\em Riemannian approximation of finsler metrics}, Asymptot. Anal., 31 (2002), pp.~177--187.

\bibitem{carrillo2011global}
{\sc J.~A. Carrillo, M.~DiFrancesco, A.~Figalli, T.~Laurent, and D.~Slep{\v{c}}ev}, {\em Global-in-time weak measure solutions and finite-time aggregation for nonlocal interaction equations},  (2011).

\bibitem{carrillo2025positivity}
{\sc J.~A. Carrillo, H.~Liu, and H.~Yu}, {\em Positivity-preserving and energy-dissipating discontinuous galerkin methods for nonlinear nonlocal {F}okker--{P}lanck equations}, Commun. Appl. Ind. Math., 16 (2025).

\bibitem{carrillo2003kinetic}
{\sc J.~A. Carrillo, R.~J. McCann, and C.~Villani}, {\em Kinetic equilibration rates for granular media and related equations: entropy dissipation and mass transportation estimates}, Rev. Mat. Iberoam., 19 (2003), pp.~971--1018.

\bibitem{carrillo2006contractions}
{\sc J.~A. Carrillo, R.~J. McCann, and C.~Villani}, {\em Contractions in the 2-{W}asserstein length space and thermalization of granular media}, Arch. Ration. Mech. Anal., 179 (2006), pp.~217--263.

\bibitem{coudreuse2025li}
{\sc F.~Coudreuse}, {\em Li-{Y}au-{H}amilton inequality on the {JKO} scheme for the granular-medium equation}, arXiv preprint arXiv:2510.09231,  (2025).

\bibitem{dacorogna2008direct}
{\sc B.~Dacorogna}, {\em Direct methods in the calculus of variations}, Springer, 2008.

\bibitem{dal2012introduction}
{\sc G.~Dal~Maso}, {\em An introduction to $\Gamma$-convergence}, vol.~8, Springer Science \& Business Media, 2012.

\bibitem{erbar2010heat}
{\sc M.~Erbar}, {\em The heat equation on manifolds as a gradient flow in the {W}asserstein space}, Annales de l'IHP Probabilit{\'e}s et statistiques, 46 (2010), pp.~1--23.

\bibitem{forkert2022evolutionary}
{\sc D.~Forkert, J.~Maas, and L.~Portinale}, {\em Evolutionary {$\Gamma$}-convergence of entropic gradient flow structures for {F}okker--{P}lanck equations in multiple dimensions}, SIAM J. Math. Anal., 54 (2022), pp.~4297--4333.

\bibitem{gangbo2012homogenization}
{\sc W.~Gangbo and A.~Tudorascu}, {\em Homogenization for a class of integral functionals in spaces of probability measures}, Adv. Math., 230 (2012), pp.~1124--1173.

\bibitem{gao2023homogenization}
{\sc Y.~Gao and N.~K. Yip}, {\em Homogenisation of {Wasserstein} gradient flows}, European J. Appl. Math., 37 (2026), p.~586–613.

\bibitem{gladbach2020scaling}
{\sc P.~Gladbach, E.~Kopfer, and J.~Maas}, {\em Scaling limits of discrete optimal transport}, SIAM J. Math. Anal., 52 (2020), pp.~2759--2802.

\bibitem{jacobs2021back}
{\sc M.~Jacobs, W.~Lee, and F.~L{\'e}ger}, {\em The back-and-forth method for {W}asserstein gradient flows}, ESAIM Control Optim. Calc. Var., 27 (2021), p.~28.

\bibitem{jacobs2020fast}
{\sc M.~Jacobs and F.~L{\'e}ger}, {\em A fast approach to optimal transport: The back-and-forth method}, Numer. Math., 146 (2020), pp.~513--544.

\bibitem{jordan1998variational}
{\sc R.~Jordan, D.~Kinderlehrer, and F.~Otto}, {\em The variational formulation of the {F}okker--{P}lanck equation}, SIAM J. Math. Anal., 29 (1998), pp.~1--17.

\bibitem{jost2005riemannian}
{\sc J.~Jost}, {\em Riemannian geometry and geometric analysis}, Springer, 2005.

\bibitem{lee2015jordan}
{\sc P.~W. Lee}, {\em On the {J}ordan--{K}inderlehrer--{O}tto scheme}, J. Math. Anal. Appl., 429 (2015), pp.~131--142.

\bibitem{lisini2009nonlinear}
{\sc S.~Lisini}, {\em Nonlinear diffusion equations with variable coefficients as gradient flows in {W}asserstein spaces}, ESAIM Control Optim. Calc. Var., 15 (2009), pp.~712--740.

\bibitem{liu2025variational}
{\sc H.~Liu and A.~E. Tzavaras}, {\em Variational structure of {F}okker-{P}lanck equations with variable mobility}, arXiv preprint arXiv:2505.10676,  (2025).

\bibitem{liu2018positivity}
{\sc J.-G. Liu, L.~Wang, and Z.~Zhou}, {\em Positivity-preserving and asymptotic preserving method for 2d {K}eller-{S}egal equations}, Math. Comput., 87 (2018), pp.~1165--1189.

\bibitem{marcellini1978periodic}
{\sc P.~Marcellini}, {\em Periodic solutions and homogenization of nonlinear variational problems}, Ann. Mat. Pura Appl., 117 (1978), pp.~139--152.

\bibitem{otto2001geometry}
{\sc F.~Otto}, {\em The geometry of dissipative evolution equations: the porous medium equation},  (2001).

\bibitem{pavliotis2008multiscale}
{\sc G.~A. Pavliotis and A.~Stuart}, {\em Multiscale methods: averaging and homogenization}, vol.~53, Springer Science \& Business Media, 2008.

\bibitem{pechstein2013weighted}
{\sc C.~Pechstein and R.~Scheichl}, {\em Weighted {P}oincar{\'e} inequalities}, IMA J. Numer. Anal., 33 (2013), pp.~652--686.

\bibitem{rankin2024jko}
{\sc C.~Rankin and T.-K.~L. Wong}, {\em {JKO} schemes with general transport costs}, arXiv preprint arXiv:2402.17681,  (2024).

\bibitem{sandier2004gamma}
{\sc E.~Sandier and S.~Serfaty}, {\em Gamma-convergence of gradient flows with applications to {G}inzburg-{L}andau}, Commun. Pure Appl. Math., 57 (2004), pp.~1627--1672.

\bibitem{santambrogio2024strong}
{\sc F.~Santambrogio and G.~Toshpulatov}, {\em Strong {$L^2 H^2$} convergence of the {JKO} scheme for the {F}okker--{P}lanck equation}, Arch. Ration. Mech. Anal., 248 (2024), p.~99.

\bibitem{serfaty2011gamma}
{\sc S.~Serfaty}, {\em Gamma-convergence of gradient flows on {H}ilbert and metric spaces and applications}, Discrete Contin. Dyn. Syst, 31 (2011), pp.~1427--1451.

\bibitem{strikwerda2004finite}
{\sc J.~C. Strikwerda}, {\em Finite difference schemes and partial differential equations}, SIAM, 2004.

\bibitem{struwe2000variational}
{\sc M.~Struwe}, {\em Variational methods: applications to nonlinear partial differential equations and {H}amiltonian systems}, Springer, 2008.

\bibitem{villani2008optimal}
{\sc C.~Villani}, {\em Optimal transport: old and new}, vol.~338, Springer, 2008.

\bibitem{weinan1991class}
{\sc E.~Weinan}, {\em A class of homogenization problems in the calculus of variations}, Commun. Pure Appl. Math., 44 (1991), pp.~733--759.

\end{thebibliography}

\end{document}